\documentclass[reqno,11pt]{amsart} 

\usepackage[utf8]{inputenc}
\usepackage{amsfonts,amssymb,amsmath,amsthm,latexsym,epsfig,amscd,bbm,stmaryrd,mathrsfs}
\usepackage[english]{babel}
\usepackage{graphicx}

\usepackage{graphicx}
\usepackage{color}
\usepackage{tikz}
\usepackage{courier}
\usepackage{bbm}
\usepackage{enumerate}
\usepackage{psfrag}
\usepackage{wasysym}
\usepackage{type1cm}
\usepackage{bookmark}

\newcommand{\twoeqref}[2]{(\ref{#1}--\ref{#2})}

\newcommand{\ssup}[1] {{{\scriptscriptstyle{({#1}})}}} 
\usepackage{calc}

\def\arraypar#1{\parbox[c]{\textwidth - 2cm}{\centering #1}}
\usepackage{environ}
\NewEnviron{display}{
\begin{equation}\begin{array}{c}
\arraypar{\BODY}
\end{array}\end{equation}
}

\def\clap#1{\hbox to 0pt{\hss#1\hss}}

\makeatletter \@addtoreset{equation}{section}
\makeatother

\makeatletter \@addtoreset{enunciato}{section}
\makeatother

\newcounter{enunciato}[section]

\newtheorem{ittheorem}{Theorem}
\newtheorem{itlemma}{Lemma}
\newtheorem{itproposition}{Proposition}
\newtheorem{itdefinition}{Definition}
\newtheorem{itremark}{Remark}
\newtheorem{itclaim}{Claim}
\newtheorem{itfact}{Fact}
\newtheorem{itconjecture}{Conjecture}
\newtheorem{itcorollary}{Corollary}

\newenvironment{theorem}{\addtocounter{enunciato}{1}
\begin{ittheorem}}{\end{ittheorem}}

\newenvironment{lemma}{\addtocounter{enunciato}{1}
\begin{itlemma}}{\end{itlemma}}

\newenvironment{proposition}{\addtocounter{enunciato}{1}
\begin{itproposition}}{\end{itproposition}}

\newenvironment{definition}{\addtocounter{enunciato}{1}
\begin{itdefinition}}{\end{itdefinition}}

\newenvironment{remark}{\addtocounter{enunciato}{1}
\begin{itremark}}{\end{itremark}}

\newenvironment{conjecture}{\addtocounter{enunciato}{1}
\begin{itconjecture}}{\end{itconjecture}}

\newenvironment{corollary}{\addtocounter{enunciato}{1}
\begin{itcorollary}}{\end{itcorollary}}

\newcommand{\be}[1]{\begin{equation}\label{#1}}
\newcommand{\ee}{\end{equation}}

\newcommand{\bl}[1]{\begin{lemma}\label{#1}}
\newcommand{\el}{\end{lemma}}

\newcommand{\br}[1]{\begin{remark}\label{#1}}
\newcommand{\er}{\end{remark}}

\newcommand{\bt}[1]{\begin{theorem}\label{#1}}
\newcommand{\et}{\end{theorem}}

\newcommand{\bd}[1]{\begin{definition}\label{#1}}
\newcommand{\ed}{\end{definition}}

\newcommand{\bp}[1]{\begin{proposition}\label{#1}}
\newcommand{\ep}{\end{proposition}}

\newcommand{\bc}[1]{\begin{corollary}\label{#1}}
\newcommand{\ec}{\end{corollary}}

\newcommand{\bcj}[1]{\begin{conjecture}\label{#1}}
\newcommand{\ecj}{\end{conjecture}}

\newcommand{\bpr}{\begin{proof}}
\newcommand{\epr}{\end{proof}}

\DeclareMathOperator\supp{supp}
\DeclareMathOperator\dist{dist}
\DeclareMathOperator\diam{diam}

\newcommand{\scrP}{\mathscr{P}}

\newcommand{\BB}{\mathcal B}
\newcommand{\CC}{\mathcal C}

\newcommand{\LL}{\mathcal L}
\newcommand{\MM}{\mathcal M}
\newcommand{\NN}{\mathcal N}

\newcommand{\YY}{\mathcal Y}

\def\Z{\mathbb{Z}}
\def\N{\mathbb{N}}
\def\R{\mathbb{R}}

\def\P{\mathbb{P}}
\def\E{\mathbb{E}}

\newcommand{\cc}{{\text{\rm c}}}
\newcommand{\texte}{\text{\rm e}}

\def \GW {{\mathcal{G}\mathcal{W}}} 
\def \CM {{\mathcal{C}\mathcal{M}}}

\def \UG {\mathcal{U}\mathcal{G}}

\newcommand{\mr}{\mathring}
\newcommand{\textd}{\text{\rm d}\mkern0.5mu}

\def\1{{\mathchoice {1\mskip-4mu\mathrm l}
{1\mskip-4mu\mathrm l} 
{1\mskip-4.5mu\mathrm l} {1\mskip-5mu\mathrm l}}}

\def \ba {\begin{array}}
\def \ea {\end{array}}

\def \P  {{\mathbb P}}
\def \E  {{\mathbb E}}

\def \cL {{\mathcal L}}

\def \cN {{\mathcal N}}
\def \cO {{\mathcal O}}
\def \cP {{\mathcal P}}

\def \cT {{\mathcal T}}
\def \GW {{\mathcal{G}\mathcal{W}}} 
\def \CM {{\mathcal{C}\mathcal{M}}}

\def \e {\mathrm{e}}
\def \ee {\mathrm{e}}
\def \dd {\mathrm{d}}

\DeclareSymbolFont{symbolsC}{U}{pxsyc}{m}{n}
\DeclareMathSymbol{\opentimes}{\mathrel}{symbolsC}{93}

\newcommand{\Prob}{\mathrm{P}}

\newcommand{\Probgr}{\mathfrak{P}}

\usepackage{array}
\newcolumntype{e}{>{\displaystyle}r @{\,} >{\displaystyle}c @{\,} >{\displaystyle}l}
\begin{document}

%%%%%%%%%%%%%%%%% TITLE PAGE %%%%%%%%%%%%%%%%%%%%%%%%%
 
\title[The parabolic Anderson model on a Galton-Watson tree]
{\large The parabolic Anderson model \\ \medskip on a Galton-Watson tree}
\author[Frank den Hollander, Wolfgang K\"onig, Renato S.\ dos Santos]{}

\maketitle 
\thispagestyle{empty} 
\vspace{0.2cm} 

\centerline
{\sc Frank den Hollander\renewcommand{\thefootnote}{1}
\footnote{Mathematical Institute, University of Leiden, Niels Bohrweg 1, 2333 CA Leiden, The Netherlands,\\ 
{\tt denholla@math.leidenuniuv.nl}}, 
Wolfgang K{\"o}nig\renewcommand{\thefootnote}{2}
\footnote{Technische Universit\"at Berlin, Stra\ss e des 17.\ Juni 136, 10623 Berlin, Germany, and 
Weierstrass Institute for Applied Analysis and Stochastics (WIAS), Mohrenstra\ss e 39, 10117 Berlin, Germany,\\ 
{\tt koenig@wias-berlin.de}}, 
Renato S.~dos Santos\renewcommand{\thefootnote}{3}
\footnote{Department of Mathematics, Federal University of Minas Gerais (UFMG), Av.\ Pres.\ 
Ant\^onio Carlos 6627, Belo Horizonte, Brazil,\\
{\tt rsantos@mat.ufmg.br}}}

\medskip
\centerline{University of Leiden, TU Berlin and WIAS Berlin, UFMG}

\vspace{0.4cm}

%%%%%%%%%%%%%%%%%%%%% DEDICATION

\begin{center}
\emph{We dedicate this work to Vladas Sidoravicius, \\ who was a leading light in our understanding of disordered systems.}
\end{center}

%%%%%%%%%%%%%% ABSTRACT %%%%%%%%%%%%%%%%%%%%%%%%%%

\begin{abstract}

We study the long-time asymptotics of the total mass of the solution to the parabolic Anderson model (PAM) on a supercritical Galton-Watson random tree with bounded degrees. We identify the second-order contribution to this asymptotics in terms of a variational formula that gives information about the local structure of the region where the solution is concentrated. The analysis behind this formula suggests that, under mild conditions on the model parameters, concentration takes place on a tree with minimal degree. Our approach can be applied to locally tree-like finite random graphs, in a coupled limit where both time and graph size tend to infinity. As an example, we consider the configuration model, i.e., uniform simple random graphs with a prescribed degree sequence.



\medskip\noindent
{\bf Acknowledgment:} We thank Remco van der Hofstad for helpful discussions, and an anonymous referee for useful suggestions. 
FdH was supported by the Netherlands Organisation for Scientific Research through NWO Gravitation Grant NETWORKS-024.002.003 and by the Alexander von Humboldt Foundation. WK and RSdS were supported by the German Research Foundation through DFG SPP 1590 {\it Probabilistic Structures in Evolution}.
RSdS thanks the Pró-Reitoria de Pesquisa da Universidade Federal de Minas Gerais.
\end{abstract}

\maketitle 

%%%%%%%%%%%%%% SECTION 1 %%%%%%%%%%%%%%%%%%%%%%%

\section{Introduction and main results}
\label{s:intro}

In Section~\ref{sec-introPAM} we give a brief introduction to the parabolic Anderson model. In Section \ref{s:PAM} we give the basic notation. In Sections~\ref{ss:GW} and \ref{ss:CM} we present our results for Galton-Watson trees and for the configuration model, respectively. In Section~\ref{ss:discussion} we discuss these results.

%%%

\subsection{The PAM and intermittency}
\label{sec-introPAM}

The \emph{parabolic Anderson model (PAM)} concerns the Cauchy problem for the heat equation with a random potential, 
i.e., solutions $u$ to the equation
\begin{equation}
\partial_t u(t,x) = \Delta u(t,x) + \xi(x) u(t,x) , \qquad t>0, \, x \in \mathscr{X},
\end{equation}
where $\mathscr{X}$ is a space equipped with a Laplacian $\Delta$, and $\xi$ is a random potential on $\mathscr{X}$. The operator $\Delta + \xi$ is called the \emph{Anderson operator}. Although $\Z^d$ and $\R^d$ are the most common choices for $\mathscr{X}$, other spaces are interesting as well, such as Riemannian manifolds or discrete graphs. In the present paper we study the PAM on \emph{random graphs}. For surveys on the mathematical literature on the PAM until 2016, we refer the reader to \cite{A16,K16}.

The main question of interest in the PAM is a detailed description of the concentration effect called {\it intermittency}: in the limit of large time the solution $u$ concentrates on small and well-separated regions in space, called \emph{intermittent islands}. This concentration effect can be studied particularly well in the PAM because efficient mathematical tools are available, such as eigenvalue expansions and the Feynman-Kac formula. In particular, these lead to a detailed description of the \emph{locations} of the intermittent islands, as well as the \emph{profiles} of the potential $\xi$ and the solution $u$ inside these islands.

The analysis of intermittency usually starts with a computation of the logarithmic large-time asymptotics of the total mass, encapsulated in {\em Lyapunov exponents}. There is an important distinction between the {\em annealed} setting (i.e., averaged over the random potential) and the {\em quenched} setting (i.e., almost surely with respect to the random potential). Often both types of Lyapunov exponents admit explicit descriptions in terms of \emph{characteristic variational formulas} that contain information about how the mass concentrates in space, and serve as starting points for deeper investigations. The \lq annealed\rq\ and the \lq quenched\rq\ variational formula are typically connected, but take two different points of view. They contain two parts: a rate function term that identifies which profiles of the potential are most favourable for mass concentration, and a spectral term that identifies which profiles the solution takes inside the intermittent islands.

From now on, we restrict to \emph{discrete} spaces and to random potentials that consist of i.i.d.\ variables. For $\Z^d$, the above intermittent picture was verified for several classes of marginal distributions. It turned out that the {\em double-exponential distribution} with parameter $\varrho \in (0,\infty)$, given by
\begin{equation}
\label{e:DEexact}
\Prob(\xi(0) > u) = \ee^{-\ee^{u/\varrho}}, \qquad u \in \R,
\end{equation}
is particularly interesting, because it leads to non-trivial intermittent islands and to interesting profiles of both potential and solution inside. There are four different classes of potentials, distinguished by the type of variational formula that emerges and the scale of the diameter of the intermittent island (cf.\ \cite{HKM06}). The double-exponential distribution is critical in the sense that the intermittent islands neither grow nor shrink with time, and therefore represents a class of its own. 

The setup of the present paper contains two features that are novel in the study of the PAM: (1) we consider a \emph{random} discrete space, thereby introducing another layer of randomness into the model; (2) this space has a \emph{non-Euclidean} topology, in the form of an \emph{exponential growth} of the volume of balls as a function of their radius. As far as we are aware, the discrete-space PAM has so far been studied only on $\Z^d$ and on two examples of finite deterministic graphs: the \emph{complete graph} with $n$ vertices \cite{FM90} and the $N$-dimensional \emph{hypercube} with $n=2^N$ vertices \cite{AGH16}. These graphs have unbounded degrees as $n\to\infty$, and therefore the Laplace operator was equipped with a prefactor that is equal to the inverse of the degree, unlike the Laplace operator considered here. 

Our main target is the PAM on a Galton-Watson tree with bounded degrees. However, our approach also applies to large finite graphs that are \emph{sparse} (e.g.\ bounded degrees) and {\em locally tree-like} (rare loops). As an illustration, we consider here the \emph{configuration model} or, more precisely, the \emph{uniform simple random graph} with prescribed degree sequence. We choose to work in the almost-sure (or large-probability) setting with respect to the randomnesses of both graph \emph{and} potential, and we take as initial condition a unit mass at the root of the graph. We identify the leading order large-time asymptotics of the total mass, and derive a variational formula for the correction term. This formula contains a \emph{spatial part} (identifying the subgraph on which the concentration takes place) and a \emph{profile part} (identifying the shape on that subgraph of both the potential and the solution). Both parts are new. In some cases we can identify the minimiser of the variational formula. As in the case of $\Z^d$, the structure of the islands does not depend on time: no spatial scaling is necessary.

%%%%%%%%

\subsection{The PAM on a graph}
\label{s:PAM}

We begin with some definitions and notations, and refer the reader to \cite{A16,K16} for more background on the PAM in the case of $\Z^d$.

Let $G = (V,E)$ be a \emph{simple undirected} graph, either finite or countably infinite. Let $\Delta_G$ be the Laplacian on $G$, i.e., 
\begin{equation}
(\Delta_G f)(x) := \sum_{ {y\in V:} \atop { \{x,y\} \in E} } [f(y) - f(x)], \qquad x \in V,\,f\colon\,V\to\R.
\end{equation}
Our object of interest is the non-negative solution of the Cauchy problem for the heat equation with potential $\xi\colon\,V \to \R$ and localised initial condition,
\begin{equation}
\label{e:PAMdef}
\begin{array}{llll}
\partial_t u(x,t) &=& (\Delta_G u)(x,t) + \xi(x) u(x,t), &x \in V,\, t>0,\\
u(x,0) &=& \delta_\cO(x), &x \in V,
\end{array}
\end{equation}
where $\cO\in V$ is referred to as the \emph{origin} or \emph{root} of $G$. We say that $G$ is \emph{rooted at }$\cO$ and call $G=(V,E,\cO)$ a \emph{rooted graph}. The quantity $u(t,x)$ can be interpreted as the amount of mass present at time $t$ at site $x$ when initially there is unit mass at $\cO$. 

Criteria for existence and uniqueness of the non-negative solution to \eqref{e:PAMdef} are well-known for the case $G=\Z^d$ (see \cite{GM90}), and rely on the \emph{Feynman-Kac formula}
\begin{equation}
\label{e:FK}
u(x,t) = \E_\cO \left[ \exp \left\{\int_0^t \xi(X_s) \textd s \right\}\,\1\{X_t = x\} \right],
\end{equation}
where $X=(X_t)_{t \geq 0}$ is the continuous-time random walk on the vertices $V$ with jump rate $1$ along the edges $E$, and $\P_\cO$ denotes the law of $X$ given $X_0=\cO$. We will be interested in the \emph{total mass} of the solution, 
\begin{equation}\label{e:mass} 
U(t):= \sum_{x\in V} u(x,t) = \E_\cO \left[ \exp \left\{\int_0^t \xi(X_s) \textd s \right\}\right].
\end{equation}
Often we suppress the dependence on $G,\xi$ from the notation. Note that, by time reversal and the linearity of \eqref{e:PAMdef}, $U(t) = \hat{u}(0,t)$ with $\hat{u}$ the solution with a different initial condition, namely, constant and equal to $1$.

Throughout the paper, we assume that the random potential $\xi = (\xi(x))_{x \in V}$ consists of i.i.d.\ random variables satisfying:

\medskip\noindent
{\bf Assumption (DE).}
For some $\varrho \in (0,\infty)$,
\begin{equation}
\label{e:DE}
\Prob \left( \xi(0) \geq 0\right) = 1, \qquad \Prob \left( \xi(0) > u \right) = \ee^{-\ee^{u/\varrho}} \;\; 
\text{for } u \text{ large enough.}
\end{equation}

\medskip\noindent
Under Assumption (DE), $\xi(0) \geq 0$ almost surely and $\xi(x)$ has an eventually exact double-exponential upper tail. The latter restrictions are helpful to avoid certain technicalities that are unrelated to the main message of the paper and that require no new ideas. In particular, \eqref{e:DE} is enough to guarantee existence and uniqueness of the non-negative solution to \eqref{e:PAM} on any discrete graph with at most exponential growth, as can be inferred from the proof of the $\Z^d$-case in \cite{GM98}. All our results remain valid under \eqref{e:DEexact} or even milder conditions, e.g.\ \cite[Assumption~(F)]{GM98} plus an integrability condition on the lower tail of $\xi(0)$.

The following \emph{characteristic variational problem} will turn out to be important for the description of the asymptotics of $U(t)$ when $\xi$ has a double-exponential tail. Denote by $\cP(V)$ the set of probability measures on $V$. For $p \in \cP(V)$, define
\begin{equation}
\label{e:defIJ}
I_E(p) := \sum_{\{x,y\} \in E} \left( \sqrt{p(x)} - \sqrt{p(y)}\,\right)^2,
\qquad J_V(p) := - \sum_{x \in V} p(x) \log p(x),
\end{equation}
and set
\begin{equation}
\label{e:defchiG}
\chi_G(\varrho) := \inf_{p \in \cP(V)} [I_E(p) + \varrho J_V(p)], \qquad \varrho \in (0,\infty).
\end{equation}
The first term in \eqref{e:defchiG} is the quadratic form associated with the Laplacian, describing the solution $u(\cdot,t)$ in the intermittent islands, while the second term in \eqref{e:defchiG} is the Legendre transform of the rate function for the potential, describing the highest peaks of $\xi(\cdot)$ inside the intermittent islands. See Section \ref{ss:discussion} for its relevance and interpretation, and Section \ref{ss:chi} for alternate representations.

%%%

\subsection{Results: Galton-Watson Trees}
\label{ss:GW}

In this section we focus on our first example of a random graph.

Let $D_0$, $D_g$ be random variables taking values in $\N=\{1,2,3,\dots\}$. The \emph{Galton-Watson tree} with initial degree distribution $D_0$ and general degree distribution $D_g$ is constructed as follows. Start with a root vertex $\cO$, and attach edges from $\cO$ to $D_0$ first-generation vertices. Proceed recursively: after having attached the $n$-th generation of vertices, attach to each one of them an independent $(D_g-1)$-distributed number of new vertices, whose union gives the $(n+1)$-th generation of vertices. Denote by $\GW=(V,E)$ the graph obtained, by $\Probgr$ its probability law, and by $E$ the corresponding expectation. The law of $D_g-1$ is the offspring distribution of $\GW$, and the law of $D_g$ is the degree distribution. Write $\supp(D_g)$ to denote the set of degrees that are taken by $D_g$ with positive probability.

We will work under the following bounded-degree assumption:

\medskip\noindent
{\bf Assumption (BD).}
\begin{equation}
\label{e:GWbdddeg1}
d_{\min} := \min \supp(D_g) \geq 2, \qquad E[D_g] > 2,
\end{equation}
and, for some $d_{\max} \in \N$, $d_{\max} \geq d_{\min}$,
\begin{equation}
\label{e:GWbdddeg2}
\max \supp(D_g) \leq d_{\max}.
\end{equation}

\medskip\noindent
Under Assumption~(BD), $\GW$ is almost surely an infinite tree. Moreover,
\begin{equation}
\label{e:volumerateGW}
\lim_{r \to \infty} \frac{\log |B_r(\cO)|}{r} = \log E[D_g-1] =: \vartheta > 0 \qquad \Probgr-a.s.,
\end{equation}
where $B_r(\cO)$ is the ball of radius $r$ around $\cO$ in the graph distance (see e.g.\ \cite[pp.134--135]{LP16}). Note that Assumption (BD) allows deterministic trees with constant offspring $d_{\rm min}-1$ (provided $d_{\rm min}\geq 3$).

To state our main result, we define the constant
\begin{equation}
\label{e:deftildechi}
\widetilde{\chi}(\varrho) := \inf \big\{ \chi_T(\varrho) \colon\, T \text{ infinite tree with degrees in } \supp(D_g) \big\}
\end{equation}
with $\chi_G(\varrho)$ defined in \eqref{e:defchiG}.

\begin{theorem}{\bf [Quenched Lyapunov exponent for the PAM on $\GW$]}
\label{t:QLyapGWT}
Let $G=\GW=(V,E,\cO)$ be the rooted Galton-Watson random tree satisfying Assumption (BD), and let $\vartheta$ be as in \eqref{e:volumerateGW}. Let $\xi=(\xi(x))_{x\in V}$ be an i.i.d.\ potential satisfying Assumption (DE). Let $U(t)$ denote the total mass at time $t$ of the solution $u$ to the PAM on $\GW$. Then, as $t\to \infty$,
\begin{equation}
\label{e:QLyapGWT}
\frac{1}{t} \log U(t) = \varrho \log \left(\frac{\varrho t \vartheta}{\log\log t}\right) 
-\varrho - \widetilde{\chi}(\varrho) + o(1),
\qquad (\Prob \times \Probgr)\text{-a.s.}
\end{equation}
\end{theorem}

\noindent
The proof of Theorem~\ref{t:QLyapGWT} is given in Section~\ref{s:proofoutline}.

For $\varrho$ sufficiently large we can identify the infimum in \eqref{e:deftildechi}. 
For $d \ge 2$, denote by $\cT_d$ the \emph{infinite homogeneous tree} with degree equal to $d$ at every node.

\begin{theorem}{\bf [Identification of the minimiser]}
\label{t:tildechilargerho}
If $\varrho \geq 1/\log(d_{\rm min}+1)$, then $\widetilde{\chi}(\varrho) = \chi_{\cT_{d_{\min}}}(\varrho)$.
\end{theorem}

\noindent
The proof of Theorem~\ref{t:tildechilargerho} is given in Section \ref{s:analysischi} with the help of a comparison argument that appends copies of the infinite $d_{\rm min}$-tree to itself. We believe $\cT_{d_{\min}}$ to be the \emph{unique} minimizer of \eqref{e:deftildechi} under the same assumptions, but proving so would require more work.

%%%

\subsection{Results: Configuration Model}
\label{ss:CM}

In this section we focus on our second example of a random graph.

For $n \in \N$, let $\mathfrak{d}^{\ssup n} = (d_i^{\ssup n})_{i=1}^n$ be a collection of positive integers. The \emph{configuration model with degree sequence} $\mathfrak{d}^{\ssup n}$ is a random \emph{multigraph} (i.e., a graph that may have self-loops and multiple edges) on the vertex set $V_n := \{1, \ldots, n\}$ defined as follows. To each $i \in V_n$, attach $d_i^{\ssup n}$ \lq half-edges\rq. After that, construct edges by successively attaching each half-edge uniformly at random to a remaining half-edge. For this procedure to be successful, we must require that
\begin{equation}
\label{e:eventotdeg}
d_1^{\ssup n} + \cdots + d_n^{\ssup n} \text{ is even for every } n \in \N.
\end{equation}
Draw a root $\cO_n$ uniformly at random from $V_n$. Denote by $\CM_n = (V_n, E_n, \cO_n)$ the rooted multigraph thus obtained, and by $\Probgr_n$ its probability law. For further details, we refer the reader to \cite[Chapter~7]{vdH17a}.

We will work under the following assumption on $\mathfrak{d}^{\ssup n}$:

\medskip\noindent
{\bf Assumption~(CM):}
The degree sequences $\mathfrak{d}^{\ssup n} = (d_i^{\ssup n})_{i=1}^n$, $n\in\N$, satisfy \eqref{e:eventotdeg}. Moreover,
\begin{enumerate}
\item 
There exists an $\N$-valued random variable $D$ such that $d_{\cO_n}^{\ssup n} \Rightarrow D$ as $n\to\infty$.
\item 
$d_{\min} := \min \supp(D) \geq 3$.
\item 
There exists a $d_{\max} \in \N$ such that $2 \leq d_i^{\ssup n} \leq d_{\max}$ for all $n \in \N$  and $1 \leq i \leq n$.
\end{enumerate}

\medskip\noindent
In particular, $3\leq d_{\rm min}\leq d_{\rm max}<\infty$ and $D \leq d_{\max}$ almost surely. It is possible to take $\mathfrak{d}^{\ssup n} $ random. In that case Assumption~(CM) must be required almost surely or in probability with respect to the law of $\mathfrak{d}^{\ssup n} $, and our results below must be interpreted accordingly.

\begin{proposition}{\bf [Connectivity and simplicity of $\CM_n$]}
\label{p:CMsimpleconnected}
Under Assumption~(CM),
\begin{equation}
\label{e:CMprobsimp}
\lim_{n \to \infty} \Probgr_n( \CM_n \text{ is a simple graph}) = \ee^{-\frac{\nu}{2}- \frac{v^2}{4}},
\end{equation}
where
\begin{equation}
\label{e:defnu}
\nu := \frac{E[D(D-1)]}{E[D]} \in [2,\infty).
\end{equation}
Moreover,
\begin{equation}
\label{e:CMprobconnect}
\lim_{n \to \infty} \Probgr_n\big(\CM_n \text{ is connected} \mid \CM_n \text{ is simple}\big) = 1.
\end{equation}
\end{proposition}

\begin{proof}
See \cite[Theorem~7.12]{vdH17a} and \cite[Theorem~2.3]{FvdH17}.
\end{proof}

Item \eqref{e:CMprobsimp} in Proposition~\ref{p:CMsimpleconnected} tells us that for large $n$ the set
\begin{equation}
\mathscr{U}_n(\mathfrak d^{\ssup n}):= \left\{ \text{simple graphs on }
\{1,\dots,n\}\text{ with degrees }  d_1^{\ssup n},\dots,d_n^{\ssup n} \right\}
\end{equation}
is non-empty. Hence, we may consider the \emph{uniform simple random graph} $\UG_n$ that is drawn uniformly at random from $\mathscr{U}_n(\mathfrak d^{\ssup n})$.

\begin{proposition}{\bf [Conditional law of $\CM_n$ given simplicity]}
\label{p:CMisUG}
Under the conditional law $\Probgr_n(~\cdot \mid \CM_n \text{ is simple})$, $\CM_n$ has the same law as $\UG_n$.
\end{proposition}

\begin{proof}
See \cite[Proposition~7.15]{vdH17a}.
\end{proof}

As usual, for a sequence of events $(A_n)_{n \in \N}$, we say that $A_n$ occurs \emph{with high probability (whp)} as $n \to \infty$ if the probability of $A_n$ tends to $1$ as $n \to \infty$. This notion does not require the events to be defined on the same probability space. We denote by $\dist_{\rm TV}(X,Y)$ the total variation distance between two random variables $X$ and $Y$ (i.e., between their laws). Let
\begin{equation}
\label{e:defPhin}
\Phi_n := \Big( \frac1n \vee \dist_{\rm TV}(d^{\ssup n}_{\cO_n}, D) \Big)^{-1},
\end{equation}
and note that, by Assumption~(CM), $\Phi_n \to \infty$ as $n \to \infty$.

\begin{theorem}{\bf [Quenched Lyapunov exponent for the PAM on $\UG_n$]}
\label{t:QLyapCM}
For any $n\in\N$, let $G=\UG_n$ be the uniform simple random graph with degree sequence $\mathfrak d^{\ssup n}$ satisfying Assumption~(CM). For any $n\in\N$, let $\xi$ be an i.i.d.\ potential on $V_n$ satisfying Assumption~(DE). Let $U_n(t)$ denote the total mass of the solution to the PAM on $G=\UG_n$ as defined in Section \ref{s:PAM}. Fix a sequence of times $(t_n)_{n\in\N}$ with $t_n\to\infty$ and $t_n \log t_n = o(\log \Phi_n)$ as $n \to\infty$. Then, with high $\Prob \times \Probgr_n$-probability as $n \to \infty$,
\begin{equation}
\label{e:QLyapCM}
\frac{1}{t_n} \log U_n(t_n) = \varrho \log \left(\frac{\varrho t_n \vartheta}{\log\log t_n}\right) 
-\varrho - \widetilde{\chi}(\varrho) + o(1),
\end{equation}
where $\vartheta := \log \nu > 0$ with $\nu$ as in \eqref{e:defnu}, and $\widetilde{\chi}(\varrho)$ is as in \eqref{e:deftildechi}.
\end{theorem}

\noindent
The proof of Theorem~\ref{t:QLyapCM} is given in Section~\ref{s:proofThCM}. The main ingredients in the proof are Theorem~\ref{t:QLyapGWT} and a well-known comparison between the configuration model and an associated Galton-Watson tree inside a slowly-growing ball, from which the condition on $t_n$ originates.

Condition (1) in Assumption (CM) is a standard regularity condition. Conditions (2) and (3) provide easy access to results such as Propositions~\ref{p:CMsimpleconnected}--\ref{p:CMisUG} above. As examples of degree sequences satisfying Assumption~(CM) we mention:
\begin{itemize}
\item 
{\em Constant degrees.} In the case where $d_i=d \geq 3$ for a deterministic $d \in \N$ and all $1 \leq i \leq n$, we have $d_{\cO_n}=D=d$ almost surely, and $\UG_n$ is a uniform regular random graph. To respect \eqref{e:eventotdeg}, it is enough to restrict to $n$ such that $nd$ is even. In this case $\dist_{\rm TV}(d_{\cO_n}, D) = 0$, and so $\Phi_n = n$ in \eqref{e:defPhin}.
\item 
{\em Random degrees.} In the case where $(d_i)_{i \in \N}$ forms an i.i.d.\ sequence taking values in $\{3,\dots,d_{\max}\}$, classical concentration bounds (e.g.\ Azuma's inequality) can be used to show that, for any $\gamma \in (0,\tfrac12)$,
\begin{equation}
\label{e:TViidcase}
d_{\rm TV}(d_{\cO_n}, D) = o(n^{-\gamma}) \quad \text{ almost surely as } n \to \infty,
\end{equation}
and so $\Phi_n \gg n^\gamma$. The condition in \eqref{e:eventotdeg} can be easily satisfied after replacing $d_n$ by $d_n + 1$ when $d_1 + \cdots + d_n$ is odd, which does not affect \eqref{e:TViidcase}. With this change, Assumption~(CM) is satisfied. For more information about $\CM_n$ with i.i.d.\ degrees, see \cite[Chapter~7]{vdH17a}.
\end{itemize}

%%%

\subsection{Discussion}
\label{ss:discussion}

Our main results, Theorems \ref{t:QLyapGWT} and \ref{t:QLyapCM}, identify the quenched logarithmic asymptotics of the total mass of the PAM. Our proofs show that the first term in the asymptotics comes from the height of the potential in an intermittent island, the second term $-\varrho$ from the probability of a quick sprint by the random walk in the Feynman-Kac formula from $\cO$ to the island, and the third term $\widetilde\chi(\varrho)$ from the structure of the island and the profile of the potential inside. Below we explain how each of these three terms comes about. Much of what follows is well-known from the study of the PAM on $\Z^d$ (see also \cite{K16}), but certain aspects are new and derive from the randomness of the ambient space and its exponential growth.

\medskip\noindent
$\blacktriangleright$ Galton-Watson tree.

\paragraph{$\bullet$ \bf First and second terms.}
The large-$t$ asymptotics of the Feynman-Kac formula \eqref{e:FKformula} for $U(t)$ comes from those random walk paths $(X_s)_{s\in[0,t]}$ that run within $\mathfrak s_t$ time units to some favorable local region of the graph (the intermittent island) and subsequently stay in that region for the rest of the time. In order to find the scale $\mathfrak r_t$ of the distance to the region and the time $\mathfrak s_t$ of the sprint, we have to balance and optimise a number of crucial quantities: the number of sites in the ball $B_{\mathfrak r_t}(\cO)$ around $\cO$ with radius $\mathfrak r_t$, the scale of the maximal value of the potential within that ball, the probability to reach that ball within time $\mathfrak s_t$, and the gain from the Feynman-Kac formula from staying in that ball during $t-\mathfrak s_t$ time units. One key ingredient is the well-known fact that the maximum of $m$ independent random variables satisfying Assumption (DE) is asymptotically equal to $h_m \approx \varrho \log\log m$ for large $m$. Another key ingredient is that $B_{\mathfrak r_t}(\cO)$ has approximately $\e^{\mathfrak r_t  \vartheta}$ vertices (see \eqref{e:volumerateGW}). Hence, this ball contains values of the potential of height $\approx h_{\e^{\mathfrak r_t  \vartheta}}\approx \varrho\log\mathfrak (r_t \vartheta)$, not just at one vertex but on a cluster of vertices of arbitrary finite size. The contribution from staying in such as cluster during $\approx t$ time units yields the first term of the asymptotics, where we still need to identify $\mathfrak r_t$. A slightly more precise calculation, involving the probabilistic cost to run within $\mathfrak s_t$ time units over $\mathfrak r_t$ space units and to afterwards gain a mass of size $(t-\mathfrak s_t) \varrho\log\mathfrak (r_t \vartheta)$, reveals that the optimal time is $\mathfrak s_t\approx \mathfrak{r}_t/\varrho\log \mathfrak{r}_t$. Optimising this together with the first term $\varrho\log\mathfrak (r_t \vartheta)$ over $\mathfrak r_t$, we see that the optimal distance is $\mathfrak r_t=\varrho t/\log\log t$. The term $-\varrho$ comes from the probability of making $\mathfrak{r}_t$ steps within $\mathfrak s_t=\mathfrak{r}_t/\varrho\log \mathfrak{r}_t$ time units. 

\paragraph{$\bullet$ \bf Third term.}
The variational formula $\widetilde\chi_G(\varrho)$ describes the second-order asymptotics of the gain of the random walk from staying $\approx t$ time units in an optimal local region (the first-order term has already been identified as $\varrho\log\mathfrak (r_t \vartheta)$). Indeed, pick some finite tree $T$ that is admissible, i.e., has positive probability to occur locally in the graph $G=\GW$. Many copies of $T$ occur disjointly with positive density in $G$. In particular, they appear within the ball $B_{\mathfrak r_t}(\cO)$ a number of times that is proportional to the volume of the ball. By standard extreme-value analysis, on one of these many copies of $T$ the random potential achieves an approximately optimal height ($\approx \varrho\log\mathfrak (r_t \vartheta)$) and shape. The optimality of the shape is measured in terms of the negative local Dirichlet eigenvalue $-\lambda_T(\xi)$ of $\Delta_G+\xi$ inside $T$. The shapes $q$ that $\xi$ can assume locally are those that have a large-deviation rate value $\LL(q)=\sum_{x}\ee^{q(x)/\varrho}$ at most 1 (note that $\LL(q)$ measures the probabilistic cost of the shape $q$ on an exponential scale). All allowed shapes $q$ are present locally at some location inside the ball $B_{\mathfrak r_t}(\cO)$ for large $t$. Each of these locations can be used by the random walk as an intermittent island. Optimising over all allowed shapes $q$, we see that the second-order term of the long stay in that island must indeed be expressed by the term
\begin{equation}
\sup_{q\colon \LL(q)\leq 1} [-\lambda_T(q)].
\end{equation}
When $T$ is appropriately chosen, this number is close to the number $\widetilde \chi(\varrho)$ defined in \eqref{e:deftildechi} (cf.\ Proposition~\ref{p:dualrepchi}). This completes the heuristic explanation of the asymptotics in \eqref{e:QLyapGWT}.

\medskip\noindent
$\blacktriangleright$ Configuration Model.

\noindent
The analogous assertion for the configuration model in \eqref{e:QLyapCM} is understood in the same way, ignoring the fact that the graph is now finite, and that size and time are coupled. As to the additional growth constraint on $t_n\log t_n$ in Theorem \ref{t:QLyapCM}: its role is to guarantee that the ball $B_{\mathfrak r_{t_n}}(\cO)$ is small enough to contain no loop with high probability. In fact, this ball is very close in distribution to the same ball in an associated Galton-Watson tree (cf.\ Proposition~\ref{p:couplingCMGW}), which allows us to carry over our result. 

\paragraph{\bf Minimal degree tree is optimal.}
What is a heuristic explanation for our result in Theorem \ref{t:tildechilargerho} that the optimal tree is an infinitely large homogeneous tree of minimal degree $d_{\rm min}$ at every vertex? The first term in \eqref{e:defchiG}, the quadratic form associated with the Laplacian, has a spread-out effect. Apparently, the self-attractive effect of the second term is not strong enough to cope with this, as the super-linear function $p\mapsto p\log p$ in the definition of $J_V$ in \eqref{e:defIJ} is \lq weakly superlinear\rq. This suggests that the optimal structure should be infinitely large (also on $\Z^d$ the optimal profile is positive anywhere in the ambient space $\Z^d$). The first term is obviously monotone in the degree, which explains why the infinite tree with minimal degree optimises the formula. 

\paragraph{\bf Hurdles.}
The exponential growth of the graph poses a number of technical difficulties that are not present for the PAM on $\Z^d$ or $\R^d$. Indeed, one of the crucial points in the proof of the upper bound for the large-time asymptotics is to restrict the infinite graph $G$ to some finite but time-dependent subgraph (in our case the ball $B_{\mathfrak r_t}(\cO)$). On $\Z^d$, a reflection technique that folds $\Z^d$ into a box of an appropriate size gives an upper bound at the cost of a negligible boundary term. For exponentially growing graphs, however, this technique can no longer be used because the boundary of a large ball is comparable in size to the volume of the ball. Therefore we need to employ and adapt an intricate method developed on $\Z^d$ for deriving deeper properties of the PAM, namely, Poisson point process convergence of all the top eigenvalue-eigenvector pairs and asymptotic concentration in a single island. This method relies on certain \emph{path expansions}, which are developed in Section \ref{s:pathexpansions} and rely on ideas from \cite{BKS18, MP16}.

%%%

\subsection{Open questions}
\label{ss:openq}

We discuss next a few natural questions for future investigation.

\paragraph{\bf Unbounded degrees.}
A central assumption used virtually throughout in the paper is that of a uniformly bounded degree for the vertices of the graph. While this assumption can certainly be weakened, doing so would require a careful analysis of many interconnected technical arguments involving both the geometry of the graph and the behaviour the random walk.  An inspection of our proofs will reveal that some mild growth of the maximal degree with the volume is allowed, although this would not address the real issues at hand and would therefore be far from optimal. For this reason we prefer to leave unbounded degrees for future work.

\paragraph{\bf Small $\varrho$.} 
The question of whether Theorem~\ref{t:tildechilargerho} is still true when $\varrho < 1/\log(d_{\min} +1)$ seems to us not clear at all, and in fact interesting. Indeed, the analogous variational problem in $\Z^d$ was analysed in \cite{GdH99} and was shown to be highly non-trivial for small $\varrho$.

\paragraph{\bf Different time scales.}
In a fixed finite graph, the PAM can be shown to localise for large times on the site that maximises the potential. It is reasonable to expect the same when the graph is allowed to grow but only very slowly in comparison to the time window considered, leading to a behaviour very different from that shown in Theorem~\ref{t:QLyapCM}. A more exciting and still widely open question is whether there could be other growth regimes between graph size and time that would lead to new asymptotic behaviours. We expect that Theorem~\ref{t:QLyapCM} would still hold for times well above the time cutoff given. For investigations of a similar flavour we direct the reader to \cite{AGH16, FM90}.

\paragraph{\bf Annealing.}
In the present paper we only consider the \emph{quenched} setting, i.e., statements that hold almost-surely or with high probability with respect to the law of both the random graph and the random potential. There are three possible \emph{annealed} settings, where we would average over one or both of these laws. Such settings would certainly lead to different growth scales for the total mass, corresponding to new probabilities to observe local structures in the graph and/or the potential. The variational problems could be potentially different, but for double-exponential tails comparison with the $\Z^d$ case suggests that they would coincide.

%%%

\subsection{Outline}

The remainder of the paper is organised as follows. In Section \ref{s:prep} we collect some basic notations and facts about graphs, spectral objects, alternate representations of the characteristic formula $\chi_G(\varrho)$, and the potential landscape. In Section \ref{s:pathexpansions} we employ a path expansion technique to estimate the contribution to the Feynman-Kac formula coming from certain specific classes of paths. In Section \ref{s:proofoutline} we prove Theorem \ref{t:QLyapGWT}. In Section \ref{s:proofThCM} we prove Theorem \ref{t:QLyapCM}. In Appendix \ref{s:analysischi} we analyse the behavior of the variational formula $\chi_T$ for trees $T$ under certain glueing operations, and prove Theorem \ref{t:tildechilargerho}.

%%%%%%%%%%%%%%% SECTION 2 %%%%%%%%%%%%%%%%%%%%%%%%%%%%%%

\section{Preliminaries}
\label{s:prep}

In this section we gather some facts that will be useful in the remainder of the paper. In particular, we transfer some basic properties of the potential landscape derived in \cite{BK16} and \cite{BKS18} for the Euclidean-lattice setting to the {\em sparse-random-graph} setting. In Section~\ref{ss:graphs} we describe the classes of graphs we will work with. In Section \ref{ss:specbounds} we derive spectral bounds on the Feynman-Kac formula. In Section~\ref{ss:chi} we provide alternative representations for the constant $\chi$ in \eqref{e:defchiG}. In Section \ref{ss:islands} we obtain estimates on the maximal height of the potential in large balls as well as on the sizes and local eigenvalues of the islands where the potential is close to maximal. In Section \ref{ss:connectivity} we obtain estimates on the heights of the potential seen along self-avoiding paths and on the number of islands where the potential is close to maximal.

%%%

\subsection{Graphs}
\label{ss:graphs}

All graphs considered in Section~\ref{s:prep} are simple, connected and undirected, and are either finite or countably infinite. For a graph $G=(V,E)$, we denote by $\dist(x,y) = \dist_G(x,y)$ the graph distance between $x,y \in V$, and by
\begin{equation}
\label{e:defdegree}
\deg(x) = \deg_G(x) := \#\{y \in V \colon\, \{y,x\} \in E\},
\end{equation}
the degree of the vertex $x \in V$. The ball of radius $\ell>0$ around a vertex $x$ is defined as
\begin{equation}
\label{e:defBr}
B_\ell(x) = B^G_\ell(x) := \{y \in V \colon\, \dist_G(y,x) \leq \ell\}.
\end{equation}
For a rooted graph $G=(V, E, \cO)$, the distance to the root is defined as
\begin{equation}
|x| := \dist_G(x,\cO), \qquad x\in V,
\end{equation}
and we set $B_\ell := B_\ell(\cO)$, $L_\ell:= |B_\ell|$.

The classes of graphs that we will consider are as follows. Fix a parameter $d_{\max} \in \N$. For $r \in \N_0 = \N \cup \{0\}$, define
\begin{equation}
\label{e:deffrakGr}
\mathfrak{G}_r := \left\{ \substack{
\text{simple connected undirected rooted graphs }
G = (V,E, \cO) \text{ with } \\ V \text{ finite or countable, } |V| 
\geq r+1 \text{ and } \max_{x \in V} \deg_G(x) \leq d_{\max}} \right\}.
\end{equation}
Note that if $G \in \mathfrak{G}_r$, then $L_r =|B_r| \geq r+1$. Also define
\begin{equation}
\label{e:defGinfinity}
\mathfrak{G}_\infty = \bigcap_{r \in \N_0} \mathfrak{G}_r
= \left\{ \substack{ \text{simple connected undirected rooted graphs }
G = (V,E, \cO) \text{ with } \\ V \text{ countable, } |V| = \infty \text{ and } 
\max_{x \in V} \deg_G(x) \leq d_{\max}} \right\}.
\end{equation}

When dealing with infinite graphs, we will be interested in those that have an \emph{exponential growth}. Thus we define, for $\vartheta > 0$,
\begin{equation}
\label{e:defGinfvartheta}
\mathfrak{G}^{(\vartheta)}_\infty = 
\left\{ G \in \mathfrak{G}_\infty\colon\, \lim_{r \to \infty} \frac{\log L_r }{r} = \vartheta \right\}.
\end{equation}
Note that $\GW \in \mathfrak{G}^{(\vartheta)}_\infty$ almost surely, with $\vartheta$ as in \eqref{e:volumerateGW}.

%%%

\subsection{Spectral bounds}
\label{ss:specbounds}

Let $G = (V, E)$ be a simple connected graph with maximal degree $d_{\max}\in\N$, where the vertex set $V$ may be finite or countably infinite. 

We recall the Rayleigh-Ritz formula for the principal eigenvalue of the Anderson Hamiltonian. For $\Lambda \subset V$ and $q\colon\,V \to [-\infty, \infty)$, let $\lambda^{\ssup 1}_\Lambda(q; G)$ denote the largest eigenvalue of the operator $\Delta_G + q$ in $\Lambda$ with Dirichlet boundary conditions on $V\backslash\Lambda$. More precisely,
\begin{equation}
\label{e:RRformula}
\begin{aligned}
\lambda^{\ssup 1}_\Lambda(q;G) := \sup \big\{ \langle (\Delta_G + q) \phi, \phi \rangle_{\ell^2(V)} 
\colon\, \phi \in \R^{V}, \,\supp \phi \subset \Lambda, \, \|\phi\|_{\ell^2(V)}=1 \big\}.
\end{aligned}
\end{equation}
We will often omit the superscript ``$(1)$'', i.e., write $\lambda_\Lambda(q;G) = \lambda^{\ssup 1}_\Lambda(q;G)$, and abbreviate $\lambda_G(q) := \lambda_V(q; G)$. When there is no risk of confusion, we may also suppress $G$ from the notation, and omit $q$ when $q = \xi$.

Here are some straightforward consequences of the Rayleigh-Ritz formula:
\begin{enumerate}
\item 
For any $\Gamma \subset \Lambda$, 
\begin{equation}
\label{e:monot_princev}
\max_{z \in \Gamma} q(z) - d_{\max} \le \lambda^{\ssup 1}_\Gamma(q;G) 
\le \lambda^{\ssup 1}_\Lambda(q;G) \le \max_{z \in \Lambda} q(z).
\end{equation}
\item 
The eigenfunction corresponding to $\lambda^{\ssup 1}_\Lambda(q;G)$ can be taken to be non-negative.
\item 
If $q$ is real-valued and $\Gamma \subsetneq \Lambda$ are finite and connected in $G$, then the middle inequality in \eqref{e:monot_princev} is strict and the non-negative eigenfunction corresponding to $\lambda^{\ssup 1}_\Lambda(q;G)$ is strictly positive.
\end{enumerate}
In what follows we state some spectral bounds for the Feynman-Kac formula. These bounds are deterministic, i.e., they hold for any fixed realisation of the potential $\xi \in \R^{V}$.

Inside $G$, fix a finite connected subset $\Lambda \subset V$, and let $H_\Lambda$ denote the Anderson Hamiltonian in $\Lambda$ with zero Dirichlet boundary conditions on $\Lambda^c = V \backslash \Lambda$ (i.e., the restriction of the operator $H_G = \Delta_G + \xi$ to the class of functions supported on $\Lambda$). For $y \in \Lambda$, let $u^y_\Lambda$ be the solution of
\begin{equation}
\label{e:PAM}
\begin{array}{llll}
\partial_t u(x,t) &=& (H_{\Lambda} u)(x,t), &x \in \Lambda,\,t>0,\\
u(x,0) &=& \1_y(x), &x \in \Lambda,\\
\end{array}
\end{equation}
and set $U^y_\Lambda(t) := \sum_{x \in \Lambda} u^y_\Lambda(x,t)$. The solution admits the Feynman-Kac representation
\begin{equation}
\label{e:FKformula}
u^y_\Lambda(x,t) = \E_y \left[ \exp \left\{\int_0^t \xi(X_s) \textd s \right\} 
\1 \{\tau_{\Lambda^{\cc}}>t, X_t = x\} \right],
\end{equation}
where $\tau_{\Lambda^\cc}$ is the hitting time of $\Lambda^\cc$. It also admits the spectral representation
\begin{equation}
\label{e:specrepr}
u^y_\Lambda(x,t) = \sum_{k=1}^{|\Lambda|} \texte^{t \lambda^{\ssup k}_\Lambda} 
\phi_\Lambda^{\ssup k}(y) \phi^{\ssup k}_\Lambda(x),
\end{equation}
where $\lambda^{\ssup 1}_\Lambda \ge \lambda^{\ssup 2}_\Lambda \ge \cdots \ge \lambda^{\ssup{|\Lambda|}}_\Lambda$ and $\phi^{\ssup 1}_\Lambda, \phi^{\ssup 2}_\Lambda, \ldots, \phi^{\ssup{|\Lambda|}}_\Lambda$ are, respectively, the eigenvalues and the corresponding orthonormal eigenfunctions of $H_\Lambda$. These two representations may be exploited to obtain bounds for one in terms of the other, as shown by the following lemma.

\begin{lemma}{\bf [Bounds on the solution]}
\label{l:bounds_mass}
For any $y \in \Lambda$ and any $t > 0$,
\begin{multline}
\label{e:bounds_mass}
\qquad
\texte^{t \lambda^{\ssup 1}_\Lambda} \phi^{\ssup 1}_\Lambda(y)^2 
\le \E_y \left[ \texte^{\int_0^t \xi(X_s) \textd s} \1_{\{\tau_{\Lambda^\cc} > t, X_t = y\}} \right] \\
\le \E_y \left[ \texte^{\int_0^t \xi(X_s) \textd s} \1_{\{\tau_{\Lambda^\cc} > t\}} \right]
\le \texte^{t \lambda^{\ssup 1}_\Lambda} |\Lambda|^{1/2}.
\qquad
\end{multline}
\end{lemma}

\begin{proof}
The first and third inequalities follow from \twoeqref{e:FKformula}{e:specrepr} after a suitable application of Parseval's identity. The second inequality is elementary.
\end{proof}

The next lemma bounds the Feynman-Kac formula integrated up to an exit time.

\begin{lemma}{\bf [Mass up to an exit time]}
\label{l:mass_out}
For any $y \in \Lambda$ and $ \gamma  > \lambda^{\ssup 1}_\Lambda$,
\begin{equation}
\label{e:mass_out}
\E_y \left[ \exp \left\{ \int_0^{\tau_{\Lambda^\cc}} (\xi(X_s) -  \gamma )\, \textd s \right\} \right] 
\le 1 + \frac{d_{\max}|\Lambda|}{ \gamma  - \lambda^{\ssup 1}_\Lambda}.
\end{equation}
\end{lemma}

\begin{proof}
See \cite[Lemma 4.2]{GKM07}.
\end{proof}

%%%

\subsection{About the constant $\chi$}
\label{ss:chi}

We next introduce alternative representations for $\chi$ in \eqref{e:defchiG} in terms of a \lq dual\rq\ variational formula. Fix $\varrho \in (0,\infty)$ and a graph $G=(V,E)$. The functional
\begin{equation}
\label{e:cL}
\cL_V(q;\varrho) := \sum_{x \in V} \ee^{q(x)/\varrho}\in[0,\infty], \qquad q\colon\, V \to [-\infty,\infty), 
\end{equation}
plays the role of a large deviation rate function for the potential $\xi$ in $V$ (compare with \eqref{e:DE}). 
Henceforth we suppress  the superscript ``$(1)$'' from the notation for the principal eigenvalue \eqref{e:RRformula}, i.e., we write
\begin{equation}
\lambda_{\Lambda}(q;G) = \lambda^{\ssup 1}_{\Lambda}(q;G), \qquad \Lambda \subset V,
\end{equation}
and abbreviate $\lambda_G(q) =  \lambda_V(q; G)$. We also define
\begin{equation}
\label{e:defhatchi}
\widehat{\chi}_{\Lambda}(\varrho; G) := - \sup_{\substack{q\colon V \to [-\infty,\infty), 
\\ \cL_V(q;\varrho) \leq 1 }} \lambda_{\Lambda}(q;G)\in[0,\infty), 
\qquad \widehat{\chi}_G(\varrho) := \widehat{\chi}_V(\varrho;G).
\end{equation}
The condition $\cL_V(q;\varrho) \leq 1$ on the supremum above ensures that the potentials $q$ have a fair probability under the i.i.d.\ double-exponential distribution. Finally, for an infinite rooted graph $G=(V,E, \cO)$, we define
\begin{equation}
\label{e:defchi0}
\chi_G^{\ssup 0}(\rho) := \inf_{r >0} \widehat{\chi}_{B_r}(\varrho; G).
\end{equation}

Both $\chi^{\ssup 0}$ and $\widehat{\chi}$ give different representations for $\chi$.

\begin{proposition}{\bf [Alternative representations for $\chi$]}
\label{p:dualrepchi}
For any graph $G = (V,E)$ and any $\Lambda \subset V$,
\begin{equation}
\widehat{\chi}_\Lambda(\varrho;G) \geq \widehat{\chi}_V(\varrho; G) = \widehat{\chi}_G(\varrho) = \chi_G(\varrho).
\end{equation}
If $G = (V,E,\cO) \in \mathfrak{G}_\infty$, then 
\begin{equation}
\chi^{\ssup 0}_G(\varrho) = \lim_{r \to \infty} \widehat{\chi}_{B_r}(\varrho;G) = \chi_G(\varrho).
\end{equation}
\end{proposition}

\noindent
Proposition~\ref{p:dualrepchi} will be proved in Section~\ref{ss:altrepchi}.

%%%

\subsection{Potentials and islands}
\label{ss:islands}

We next consider properties of the potential landscape. Recall that $(\xi(x))_{x \in V}$ are i.i.d.\ double-exponential random variables. Set
\begin{equation}
\label{e:defa_n}
a_L := \varrho \log \log (L \vee \ee^\ee).
\end{equation}
The next lemma shows that $a_{L_r}$ is the leading order of the maximum of $\xi$ in $B_r$.

\begin{lemma}{\bf [Maximum of the potential]}
\label{l:maxpotential}
Fix $r \mapsto g_r >0$ with $\lim_{r \to \infty} g_r = \infty$. Then
\begin{equation}
\label{e:maxpotential}
\sup_{G \in \mathfrak{G}_r} \, 
\Prob \left( \left|\max_{x \in B_r} \xi(x) - a_{L_r} \right| \geq \frac{g_r}{\log L_r} \right) 
\leq \max \left\{ \frac{1}{r^2}, \ee^{-\frac{g_r}{\varrho}} \right\} \qquad \forall\, r > 2 e^2.
\end{equation}
Moreover, for any $\vartheta>0$ and any $G\in \mathfrak{G}^{(\vartheta)}_\infty$, $\Prob$-almost surely eventually as $r \to \infty$,
\begin{equation}
\label{e:asmaxpot}
\left| \max_{x \in B_r} \xi(x) - a_{L_r} \right| \leq \frac{2 \varrho \log r}{\vartheta r}.
\end{equation}
\end{lemma}
\begin{proof}
Without loss of generality, we may assume that $g_r \leq 2 \varrho \log r$. Fix $G \in \mathfrak{G}_r$ and estimate
\begin{equation}
\label{e:prmax1}
\Prob\left(\max_{x \in B_n} \xi(x) \leq a_{L_r}-\frac{g_r}{\log L_r} \right) 
= \ee^{-\frac{1}{\varrho}L_r (\log L_r) \ee^{-\frac{g_r}{\varrho \log L_r}} }
\leq \ee^{-\frac{r \log r}{\ee^2 \varrho}}
\leq \ee^{-\frac{g_r}{\varrho}},
\end{equation}
provided $r>2\ee^2$. On the other hand, using $\ee^x \geq 1+x$, $x \in \R$, we estimate
\begin{equation}
\label{e:prmax2}
\Prob \left(\max_{x \in B_n} \xi(x) \geq a_{L_r}+\frac{g_r}{\log r} \right) 
= 1 - \left(1 - \ee^{-\ee^{\log \log L_r + \frac{g_r}{\varrho \log r}}} \right)^{L_r} \leq \ee^{-\frac{g_r}{\varrho}}.
\end{equation}
Noting that the bounds above do not depend on $G$, so the case $G \in \mathfrak{G}_r$ is concluded.

For the case $G\in \mathfrak{G}^{(\vartheta)}_\infty$, let $g_r := \tfrac32 \varrho \log r$. Note that the right-hand side of \eqref{e:maxpotential} is summable over $r \in \N$, so that, by the Borel-Cantelli lemma,
\begin{equation*}
\left|\max_{x \in B_r} \xi(x) - a_{L_r} \right| < \frac{g_r}{\log L_r} < \frac{2 \varrho \log r}{\vartheta r} 
\qquad \Prob\text{-almost surely eventually as } r \to \infty.
\qedhere
\end{equation*}
\end{proof}

For a fixed rooted graph $G=(V,E, \cO) \in \mathfrak{G}_r$, we define sets of high excedances of the potential in $B_r$ as follows. Given $A>0$, let
\begin{equation}
\label{defPi}
\Pi_{r,A} = \Pi_{r,A}(\xi) := \{z \in B_r \colon\, \xi(z) > a_{L_r} - 2A\}
\end{equation}
be the set vertices in $B_r$ where the potential is close to maximal. 
For a fixed $\alpha \in (0,1)$, define
\begin{equation}
\label{e:def_Sr}
S_r := (\log r)^\alpha
\end{equation}
and set
\begin{equation}
\label{def_D_Lambda,A}
D_{r,A} = D_{r,A}(\xi) := \{z \in B_r \colon\, \dist_{G}(z, \Pi_{r,A}) \leq S_r \} \supset \Pi_{r,A},
\end{equation}
i.e., $D_{r,A}$ is the $S_r$-neighbourhood of $\Pi_{r,A}$. Let $\mathfrak{C}_{r,A}$ denote the set of all connected components of $D_{r,A}$ in $G$, which we call \emph{islands}. For $\CC \in \mathfrak{C}_{r,A}$, let
\begin{equation}
\label{defzC}
z_\CC := \textnormal{argmax}\{\xi(z) \colon\, z \in \CC\}
\end{equation}
be the point with highest potential within $\CC$. Since $\xi(0)$ has a continuous law, $z_\CC$ is $\Prob$-a.s.\ well defined
for all $\CC \in \mathfrak{C}_{r,A}$.

The next lemma gathers some useful properties of $\mathfrak{C}_{r,A}$.

\begin{lemma}{\bf [Maximum size of the islands]}
\label{l:size_comps}
For every $A > 0$, there exists $M_A \in \N$ such that the following holds. For a graph $G \in \mathfrak{G}_r$, define the event
\begin{equation}
\BB_r := \big\{ \exists\, \CC \in \mathfrak{C}_{r,A} \text{ with } |\CC \cap \Pi_{r,A}|>M_A \big\}.
\end{equation}
Then $\sum_{r \in \N_0} \sup_{G \in \mathfrak{G}_r} \Prob(\BB_r) < \infty$. In particular,
\begin{equation}
\lim_{r \to \infty} \sup_{G \in \mathfrak{G}_r} \Prob(\BB_r) = 0,
\end{equation}
and, for any fixed $G \in \mathfrak{G}_\infty$, $\Prob$-almost surely eventually as $r \to \infty$, $\BB_r$ does not occur.
Note that
\begin{equation}
\text{on $\BB_r^\cc$ all $\CC \in \mathfrak{C}_{r,A}$ satisfy:
$|\CC \cap \Pi_{r,A}| \leq M_A$, $\diam_G(\CC) \leq 2M_A S_r$, $|\CC| \leq M_A d_{\max}^{S_r}$.}
\end{equation}
\end{lemma}

\begin{proof}
The claim follows from a straightforward estimate based on \eqref{e:DE} (see \cite[Lemma 6.6]{BK16}).
\end{proof}

Apart from the dimensions, it will be also important to control the principal eigenvalues of islands in $\mathfrak{C}_{r,A}$. For this we restrict to graphs in $\mathfrak{G}^{(\vartheta)}_\infty$.

\begin{lemma}{\bf [Principal eigenvalues of the islands]}
\label{l:eigislands}
For any $\vartheta>0$ and any $G \in \mathfrak{G}^{(\vartheta)}_\infty$, $\Prob$-almost surely eventually as $r \to \infty$,
\begin{equation}
\label{e:eigislands}
\text{ all $\CC \in \mathfrak{C}_{r,A}$ satisfy: } \;\; \lambda^{\ssup 1}_\CC(\xi; G) 
\leq a_{L_r} - \widehat{\chi}_{\CC}(\varrho; G) + \varepsilon.
\end{equation}
\end{lemma}

\begin{proof}
We follow \cite[Lemma~2.11]{GM98}. Let $\varepsilon>0$, $G = (V,E,\cO) \in \mathfrak{G}^{(\vartheta)}_\infty$, and define the event
\begin{equation}
\BB_r := 
\left\{ 
\substack{
\text{there exists a connected subset } \Lambda \subset V \text{ with } \Lambda \cap B_r \neq \emptyset,  
\\  
|\Lambda| \leq M_A d_{\max}^{S_r} \text{ and } \lambda^{\ssup 1}_\Lambda(\xi; G) 
> a_{L_r} - \widehat{\chi}_\Lambda(\varrho;G) + \varepsilon
}
\right\}
\end{equation}
with $M_A$ as in Lemma~\ref{l:size_comps}.
Note that, by \eqref{e:DE},
$\ee^{\xi(x)/\varrho}$ is stochastically dominated by $C \vee E$, where $E$ is an Exp($1$) random variable and $C>0$ is a constant.
Thus, for any $\Lambda \subset V$, using \eqref{e:defhatchi}, taking $\gamma := \sqrt{\ee^{\varepsilon/\varrho}} > 1$ and applying Markov's inequality, 
we may estimate
\begin{equation}
\begin{aligned}
\Prob \left( \lambda^{\ssup 1}_\Lambda(\xi; G) > a_{L_r} - \widehat{\chi}_\Lambda(\varrho;G) + \varepsilon \right)
& \leq \Prob\left( \cL_\Lambda(\xi - a_{L_r}-\varepsilon) > 1\right) 
 = \Prob\left( \gamma^{-1} \cL_\Lambda(\xi)> \gamma \log L_r \right) \\
& \leq \ee^{-\gamma \log L_r} E[\ee^{\gamma^{-1} \cL_\Lambda(\xi)}]
\leq \ee^{-\gamma \log L_r } K_\gamma^{|\Lambda|}
\end{aligned}
\end{equation}
for some constant $K_\gamma \in (1,\infty)$.
Next note that, for any $x \in B_r$, $n \in \N$, the number of connected subsets $\Lambda \subset V$ 
with $x \in \Lambda$ and $|\Lambda|=n$ is at most $\ee^{c_\circ  n}$ for some $c_\circ = c_{\circ}(d_{\max})>0$ 
(see e.g.\ \cite[Proof of Theorem~(4.20)]{Gr99}).
Using a union bound and applying $\log L_r \sim \vartheta r$, we estimate, for some constants $c_1, c_2, c_3>0$,
\begin{equation}
\begin{aligned}
\Prob(\BB_r) 
\leq \ee^{-(\gamma-1) \log L_r} \sum_{n=1}^{\lfloor M_A d_{\max}^{S_r} \rfloor} \ee^{c_\circ n}  K_\gamma^n
\leq c_1 \exp \left\{-c_2 r + c_3 d_{\max}^{(\log r)^\alpha} \right\} \leq \ee^{-\tfrac12 c_2 r}
\end{aligned}
\end{equation}
when $r$ is large. Now the Borel-Cantelli lemma implies that, $\Prob$-almost surely eventually as $r\to \infty$, $\BB_r$ does not occur. The proof is completed by invoking Lemma~\ref{l:size_comps}.
\end{proof}

For later use, we state the consequence for $\GW$ in terms of $\widetilde{\chi}(\rho)$ in \eqref{e:deftildechi}.

\begin{corollary}{\bf [Uniform bound on principal eigenvalue of the islands]}
\label{c:eigislandsGW}
For $G = \GW$ as in Section~{\rm \ref{ss:GW}}, $\vartheta >$ as in \eqref{e:volumerateGW}, and any $\varepsilon>0$,
$\Prob \times \Probgr$-almost surely eventually as $r \to \infty$,
\begin{equation}
\label{e:eigislandsGW}
\max_{\CC \in \mathfrak{C}_{r,A}}\lambda^{\ssup 1}_\CC(\xi; G) 
\leq a_{L_r} - \widetilde{\chi}(\varrho) + \varepsilon.
\end{equation}
\end{corollary}

\begin{proof}
First note that $\GW \in \mathfrak{G}^{(\vartheta)}_\infty$ almost surely, so Lemma~\ref{l:eigislands} applies. By Lemma~\ref{l:maxpotential}, for any constant $C>0$, the maximum of $\xi$ in a ball of radius $C S_r$ around $\cO$
is of order $O(\log \log r)$. This means that $\cO$ is distant from $\Pi_{r,A}$, in particular, $\dist(\cO, D_{r,A}) \geq 2$ almost surely eventually as $r \to \infty$. For $\CC \in \mathfrak{C}_{r,A}$, let $T_{\CC}$ be the infinite tree obtained by attaching to each $x \in \partial \CC :=\{y \notin \CC \colon\, \exists z \in \CC \text{ with } z \sim y\} \not \ni \cO$  an infinite tree with constant offspring $d_{\min}-1$. Then $T_\CC$ is an infinite tree with degrees in $\supp(D_g)$ and, by Proposition~\ref{p:dualrepchi},
\begin{equation*}
\widehat{\chi}_{\CC}(\varrho ; \GW) = \widehat{\chi}_{\CC}(\varrho ; T_\CC) 
\geq \chi_{T_\CC}(\varrho) \geq \widetilde{\chi}(\varrho),
\end{equation*}
so the claim follows by Lemma~\ref{l:eigislands}.
\end{proof}

%%%

\subsection{Connectivity}
\label{ss:connectivity}

We again work in the setting of Section~\ref{ss:graphs}. We recall the following Chernoff bound for a Binomial random variable $\textnormal{Bin}(n,p)$ with parameters $n$, $p$ (see e.g.\ \cite[Lemma 5.9]{BKS18}):
\begin{equation}
\label{e:chernoffBin}
P \left( \textnormal{Bin}(n,p) \geq u\right) \leq \exp \left\{ -u \left( \log \frac{u}{np} - 1\right)\right\} 
\qquad \forall\, u > 0.
\end{equation}

\begin{lemma}{\bf [Number of intermediate peaks of the potential]}
\label{l:bound_mediumpoints}
For any $\beta \in (0,1)$ and any $\varepsilon \in (0, \beta/2)$, the following holds. For $G \in \mathfrak{G}_r$ and a self-avoiding path $\pi$ in $G$, set
\begin{equation}
\label{e:bound_mediumpoints}
N_{\pi} = N_{\pi}(\xi) :=|\{z \in \supp(\pi) \colon\, \xi(z) > (1-\varepsilon) a_{L_r} \}|.
\end{equation}
Define the event
\begin{equation}
\BB_r := \left\{ 
\substack{\text{there exists a self-avoiding path } \pi \text{ in $G$ with } \\ 
 \supp(\pi) \cap B_r \neq \emptyset, \, |\supp(\pi)| \geq (\log L_r)^{\beta}
\text{ and }  N_\pi > \frac{|\supp(\pi)|}{(\log {L_r})^\varepsilon}}
\right\}.
\end{equation}
Then $\sum_{r \in \N_0} \sup_{G \in \mathfrak{G}_r} \Prob(\BB_r) < \infty$. In particular,
\begin{equation}
\lim_{r \to \infty} \sup_{G \in \mathfrak{G}_r} \Prob(\BB_r) = 0
\end{equation}
and, for any fixed $G \in \mathfrak{G}_\infty$, $\Prob$-almost surely eventually as $r \to \infty$, all self-avoiding paths $\pi$ in $G$ with $\supp(\pi) \cap B_r \neq \emptyset$ and $|\supp(\pi)| \geq (\log L_r)^{\beta}$ satisfy $N_{\pi} \le \frac{|\supp(\pi)|}{(\log L_r)^\varepsilon}$.
\end{lemma}

\begin{proof}
Fix $\beta \in (0,1)$ and $\varepsilon \in (0,\beta/2)$. For any $G \in \mathfrak{G}_r$, \eqref{e:DE} implies
\begin{equation}
\label{e:bound_mp1}
p_r := \Prob(\xi(0) > (1-\varepsilon)a_{L_r}) = \exp\left\{-(\log L_r)^{1-\varepsilon}\right\}.
\end{equation}
Fix $x \in B_n$ and $k \in \N$. The number of self-avoiding paths $\pi$ in $B_r$ with $|\supp(\pi)|=k$ and $\pi_0 = x$ is at most $d_{\max}^k$. For such a $\pi$, the random variable $N_{\pi}$ has a Bin($p_r$, $k$)-distribution. Using \eqref{e:chernoffBin} and a union bound, we obtain
\begin{multline}
\label{e:bound_mp4}
\Prob\Bigl( \exists\, \text{ self-avoiding } \pi \text{ with } |\supp(\pi)|=k, \pi_0 = x 
\text{ and } N_{\pi} > k/ (\log L_r)^\varepsilon \Bigr) \\
\le \exp \left\{ -k \left((\log L_r)^{1-2\varepsilon} - \log d_{\max} 
- \frac{1+ \varepsilon\log\log L_r}{(\log L_r)^{\varepsilon}}\right) \right\}.
\end{multline}
Note that, since $L_r > r$ and the function $x \mapsto \log \log x / (\log x)^\varepsilon$ is eventually decreasing, for $r$ large enough and uniformly over $G \in \mathfrak{G}_r$, the expression in parentheses above is at least $\frac12 (\log L_r)^{1- 2\varepsilon}$. Summing over $k \ge (\log L_r)^\beta$ and $x \in B_r$, we get
\begin{equation}
\label{e:bound_mp5}
\begin{aligned}
&\Prob\left(\exists\, \text{ self-avoiding } \pi \text{ such that } |\supp(\pi)| \geq (\log L_r)^\beta
\text{ and } \eqref{e:bound_mediumpoints} \text{ does not hold}\right)\\
&\qquad \le 2 L_r \exp \left\{-\tfrac12 (\log L_r)^{1+\beta-2\varepsilon} \right\}
\leq c_1 \exp \left\{-c_2 (\log L_r)^{1+\delta} \right\}
\end{aligned}
\end{equation}
for some positive constants $c_1, c_2, \delta$, uniformly over $G \in \mathfrak{G}_r$. Since $L_r > r$, \eqref{e:bound_mp5} is summable in $r$ (uniformly over $G \in \mathfrak{G}_r$). The proof is concluded invoking the Borel-Cantelli lemma.
\end{proof}

A similar computation bounds the number of high exceedances of the potential.

\begin{lemma}{\bf [Number of high exceedances of the potential]}
\label{l:boundhighexceedances}
For any $A>0$ there is a $C \ge 1$ such that, for all $\delta \in (0,1)$, the following holds. For $G \in \mathfrak{G}_r$ and a self-avoiding path $\pi$ in $G$, let
\begin{equation}
N_\pi := |\{ x \in \supp(\pi) \colon\, \xi(x) > a_{L_r} - 2A \}|.
\end{equation}
Define the event
\begin{equation}
\BB_r := \left\{ 
\substack{\text{there exists a self-avoiding path } \pi \text{ in $G$ with } \\ 
 \supp(\pi) \cap B_r \neq \emptyset, \, |\supp(\pi)| \geq C (\log L_r)^{\delta}
\text{ and }  N_\pi > \frac{|\supp(\pi)|}{(\log {L_r})^\delta}}
\right\}.
\end{equation}
Then $\sum_{r \in \N_0} \sup_{G \in \mathfrak{G}_r} \Prob(\BB_r) < \infty.$ In particular,
\begin{equation}
\lim_{r \to \infty} \sup_{G \in \mathfrak{G}_r} \Prob(\BB_r) = 0
\end{equation}
and, for any fixed $G \in \mathfrak{G}_\infty$, $\Prob$-almost surely eventually as $r \to \infty$, all self-avoiding paths $\pi$ in $G$ with $\supp(\pi) \cap B_r \neq \emptyset$ and $|\supp(\pi)| \geq C (\log L_r)^{\delta}$ satisfy
\begin{equation}
\label{e:boundhighexceedances}
N_\pi = |\{ x \in \supp(\pi) \colon\, \xi(x) > a_{L_r} - 2A \}| \le \frac{|\supp(\pi)|}{(\log L_r)^\delta}.
\end{equation} 
\end{lemma}

\begin{proof}
Proceed as for Lemma~\ref{l:bound_mediumpoints}, noting that this time
\begin{equation}
\label{e:bound_he1}
p_r := \Prob\big(\xi(0) > a_{L_r} - 2A\big) = L_r^{-\epsilon}
\end{equation}
where $\epsilon =\ee^{-\frac{2A}{\varrho}}$, and taking $C > 2/\epsilon$.
\end{proof}

%%%%%%%%%%% SECTION 3 %%%%%%%%%%%%%%%%%%%%%%%%%

\section{Path expansions}
\label{s:pathexpansions}

We again work in the setting of Section~\ref{ss:graphs}. In the following, we develop a way to bound the contribution of certain specific classes of paths to the Feynman-Kac formula, similar to what is done in \cite{BKS18} in the $\Z^d$-case. In Section~\ref{ss:keyprop} we state a key proposition reducing the entropy of paths. This proposition is proved in Section~\ref{ss:proofProp} with the help of a lemma bounding the mass of an equivalence class of paths, which is stated and proved in Section~\ref{ss:equivclasspaths} and is based on ideas from \cite{MP16}. The proof of this lemma requires two further lemmas controlling the mass of the solution along excursions, which are stated and proved in Section~\ref{ss:mass_fixed_path}.   

%%%

\subsection{Key proposition}
\label{ss:keyprop}

Fix a graph $G=(V, E, \cO) \in \mathfrak{G}_r$. We define various sets of nearest-neighbour paths in $G$ as follows. For $\ell \in \N_0$ and subsets $\Lambda, \Lambda' \subset V$, put
\begin{equation}
\label{defcurlyP}
\begin{aligned}
&\scrP_\ell(\Lambda,\Lambda') := \left\{ (\pi_0, \ldots, \pi_{\ell}) \in V^{\ell+1} \colon\,
\begin{array}{ll} 
&\pi_0 \in \Lambda, \pi_{\ell} \in \Lambda',\\
&\{\pi_{i}, \pi_{i-1}\} \in E \;\forall\, 1 \le i \le \ell
\end{array}
\right\},\\
&\scrP(\Lambda, \Lambda') := \bigcup_{\ell \in \N_0} \scrP_\ell(\Lambda,\Lambda'),
\end{aligned}
\end{equation}
and set
\begin{equation}
\scrP_\ell := \scrP_\ell(V,V), \qquad \scrP := \scrP(V,V). 
\end{equation}
When $\Lambda$ or $\Lambda'$ consists of a single point, we write $x$ instead of $\{x\}$. For $\pi \in \scrP_\ell$, we set $|\pi| := \ell$. We write $\supp(\pi) := \{\pi_0, \ldots, \pi_{|\pi|}\}$ to denote the set of points visited by $\pi$.

Let $X=(X_t)_{t\ge0}$ be the continuous-time random walk on $G$ that jumps from $x \in V$ to any neighbour $y\sim x$ with rate $1$. We denote by $(T_k)_{k \in \N_0}$ the sequence of jump times (with $T_0 := 0$). For $\ell \in \N_0$, let 
\begin{equation}
\pi^{\ssup \ell}(X) := (X_0, \ldots, X_{T_{\ell}})
\end{equation}
be the path in $\scrP_\ell$ consisting of the first $\ell$ steps of $X$ and, for $t  \ge 0$, let
\begin{equation}
\label{e:defpathX0t}
\pi(X_{[0,t]}) = \pi^{\ssup{\ell_t}}(X), \quad \text{ where } \ell_t \in \N_0 \, 
\text{ satisfies } \, T_{\ell_t} \le t < T_{\ell_t+1},
\end{equation}
denote the path in $\scrP$ consisting of all the steps taken by $X$ between times $0$ and $t$.

Recall the definitions from Section~\ref{ss:islands}. For $G \in \mathfrak{G}_r$, $\pi \in \scrP$ and $A>0$, define
\begin{equation}
\label{e:deflambdaLApi}
\lambda_{r,A}(\pi) := \sup \big\{ \lambda^{\ssup 1}_\CC(\xi; G) 
\colon\, \CC \in \mathfrak{C}_{r,A}, 
\, \supp(\pi)\cap \CC \cap \Pi_{r,A} \neq \emptyset \big\},
\end{equation}
with the convention $\sup \emptyset = -\infty$. 
This is the largest principal eigenvalue among the components of $\mathfrak C_{r,A}$ in $G$ 
that have a point of high exceedance visited by the path $\pi$.

The main result of this section is the following proposition. Hereafter we abbreviate $\log^{\ssup 3} x := \log \log \log x$.

\begin{proposition}{\bf [Entropy reduction]}
\label{p:massclass}
For every fixed $d_{\max} \in \N$, there exists an $A_0 = A_0(d_{\max}) > 0$ such that the following holds. Let $\alpha \in (0,1)$ be as in \eqref{e:def_Sr} and let $\kappa\in (\alpha,1)$. For all $A > A_0$, there exists a constant $c_A = c_A(d_{\max}) >0$ such that, with probability tending to one as $r\to\infty$ uniformly over $G \in \mathfrak{G}_r$, the following statement is true: For each $x \in B_r$, each $\NN \subset \scrP(x,B_r)$ satisfying $\supp(\pi) \subset B_r$ and $\max_{1 \le \ell \le |\pi|} \dist_{G}(\pi_\ell, x) \geq (\log L_r)^\kappa$ for all $\pi \in \mathcal{N}$, and each assignment $\pi\mapsto ( \gamma_\pi , z_\pi)\in \R \times V$ satisfying
\begin{equation}
\label{e:cond_massclass1}
\gamma_\pi \ge \left(\lambda_{r,A}(\pi)  + \texte^{-S_r} \right) \vee (a_{L_r}- A) \qquad \text{ for all } \pi \in \NN
\end{equation}
and
\begin{equation}
\label{e:cond_massclass2}
z_\pi \in \supp(\pi) \cup 
\bigcup_{ \substack{\CC \in \mathfrak{C}_{r,A} \colon \\ \supp(\pi) \cap \CC \cap \Pi_{r,A} \neq \emptyset}} \CC \qquad \text{ for all } \pi \in \NN,
\end{equation}
the following inequality holds for all $t \ge 0$:
\begin{equation}
\label{e:mass_class}
\log \E_x \left[ \texte^{\int_0^t \xi(X_s) \textd s} \1_{\{\pi(X_{[0,t]}) \in \mathcal{N}\}}\right]
\le \sup_{\pi \in \mathcal{N}} \Big\{ t \gamma_\pi - \left( \log^{\ssup 3} L_r   - c_A \right) \dist_{G}(x,z_\pi) \Big\}. 
\end{equation}
Moreover, for any $G \in \mathfrak{G}_\infty$, $\Prob$-almost surely eventually as $r\to\infty$, the same statement is true.
\end{proposition}

The key to the proof of Proposition~\ref{p:massclass} in Section \ref{ss:proofProp} is Lemma~\ref{l:fixed_class} in Section \ref{ss:equivclasspaths}, whose proof depends on Lemmas \ref{l:path_eval}--\ref{l:mass_in} in Section \ref{ss:mass_fixed_path}. We emphasize that all these results are deterministic, i.e., they hold for any realisation of the potential $\xi$.

%%%

\subsection{Mass of the solution along excursions}
\label{ss:mass_fixed_path}

Fix $G=(V,E,\cO) \in \mathfrak{G}_r$. The first step to control the contribution of a path to the total mass is to control the contribution of excursions outside $\Pi_{r,A}$  (recall~\eqref{defPi}).

\begin{lemma}{\bf [Path evaluation]}
\label{l:path_eval}
For $\ell\in\N_0$, $\pi \in \scrP_\ell$ and $\gamma  > \max_{0 \leq i < |\pi|} \{\xi(\pi_i)-\deg(\pi_i)\}$,
\begin{equation}
\label{e:path_eval}
\E_{\pi_0} \left[\exp\left\{\int_0^{T_{\ell}} (\xi(X_s) -  \gamma )\, \textd s\right\} \,\middle|\, \pi^{\ssup {\ell}}(X) = \pi  \right]
= \prod_{i=0}^{\ell-1} \frac{\deg(\pi_i)}{\gamma - [\xi(\pi_i)-\deg(\pi_i)]}.
\end{equation}
\end{lemma}

\begin{proof}
The left-hand side of \eqref{e:path_eval} can be evaluated by using the fact that $T_\ell$ is the sum of $\ell$ independent Exp($\deg(\pi_i)$) random variables that are independent of $\pi^{\ssup {\ell}}(X)$. The condition on $\gamma$ ensures that all integrals are finite.
\end{proof}

For a path $\pi \in \scrP$ and $\varepsilon \in (0,1)$, we write
\begin{equation}
\label{e:def_Mpi}
M^{r,\varepsilon}_\pi := \big| \bigl\{0 \leq i < |\pi| \colon\, \xi(\pi_i) \le (1-\varepsilon)a_{L_r}\bigr\}\big|,
\end{equation}
with the interpretation that $M^{r,\varepsilon}_\pi = 0$ if $|\pi|=0$.

\begin{lemma}{\bf [Mass of excursions]}
\label{l:mass_in}
For every $A, \varepsilon>0$ there exist $c > 0$ and $n_0 \in \N$ such that, for all $r \ge n_0$, all $\gamma > a_{L_r} - A$ and all $\pi \in \scrP$ satisfying $\pi_i \notin \Pi_{r,A}$ for all $0 \leq i < \ell:=|\pi|$,
\begin{equation}
\label{e:mass_in}
\E_{\pi_0} \left[ \exp \left\{ \int_0^{T_{\ell}}(\xi(X_t) - \gamma)\, \textd s \right\} \,\middle|\, \pi^{\ssup {\ell}}(X) = \pi \right] 
\le q_A^{\ell} \texte^{\left(c -\log^{\ssup 3} L_r \right) M^{r,\varepsilon}_\pi},
\end{equation}
where $q_A := (1+A/d_{\max})^{-1}$. Note that $\pi_{\ell} \in \Pi_{r,A}$ is allowed.
\end{lemma}

\begin{proof}
By our assumptions on $\pi$ and $\gamma$, we can use Lemma~\ref{l:path_eval}. Splitting the product in the right-hand side of \eqref{e:path_eval} according to whether $\xi(\pi_i) \geq (1-\varepsilon)a_{L_r}$ or not, and using that $\xi(\pi_i) \le a_{L_r} - 2A$ for all  $0 \leq i < |\pi|$, we bound the left-hand side of \eqref{e:mass_in} by
\begin{equation}
\label{e:mass_in1}
q_A^{\ell} \left[q_A \frac{\varepsilon a_{L_r} - A}{d_{\max}}\right]^{-|\{0 \leq i < \ell \colon\, \xi(\pi_i) 
\le (1-\varepsilon)a_{L_r}\}|}.
\end{equation}
Since $a_{L_r} = \varrho \log\log L_r \geq \varrho \log \log r$, for large $r$ the number within square brackets in \eqref{e:mass_in1} is at least $q_A \varepsilon \varrho (\log\log L_r) /2d_{\max} > 1$. Hence \eqref{e:mass_in} holds with $c := \log (1 \vee 2d_{\max} (q_A \varepsilon \varrho)^{-1})$.
\end{proof}

%%%

\subsection{Equivalence classes of paths}
\label{ss:equivclasspaths}

We follow \cite[Section 6.2]{BKS18}. Note that the distance between $\Pi_{r,A}$ and $D_{r,A}^\cc$ in $G$ is at least $S_r = (\log L_r)^\alpha$.

\begin{definition}{\bf [Concatenation of paths]} {\rm (a)}
\label{def:concat}
When $\pi$ and $\pi'$ are two paths in $\scrP$ with $\pi_{|\pi|} = \pi'_0$, 
we define their \emph{concatenation} as
\begin{equation}
\label{def_concat}
\pi \circ \pi' := (\pi_0, \ldots, \pi_{|\pi|}, \pi'_1, \ldots, \pi'_{|\pi'|}) \in \scrP.
\end{equation}
Note that $|\pi \circ \pi'| = |\pi| + |\pi'|$. 

\medskip\noindent
{\rm (b)} When $\pi_{|\pi|} \neq \pi'_0$, we can still define the \emph{shifted concatenation} of $\pi$ and $\pi'$ as $\pi \circ \hat{\pi}'$, where $\hat{\pi}' := (\pi_{|\pi|}, \pi_{|\pi|}  + \pi'_1 - \pi'_0, \ldots, \pi_{|\pi|} + \pi'_{|\pi'|} - \pi'_0)$. The shifted concatenation of multiple paths is defined inductively via associativity. 
\end{definition}

Now, if a path $\pi \in \scrP$ intersects $\Pi_{r,A}$, then it can be decomposed into an initial path, 
a sequence of excursions between $\Pi_{r,A}$ and $D_{r,A}^\cc$, and a terminal path. 
More precisely, there exists $m_\pi \in \N $ such that
\begin{equation}
\label{e:concat1}
\pi = \check{\pi}^{\ssup 1} \circ \hat{\pi}^{\ssup 1} \circ \cdots \circ \check{\pi}^{\ssup {m_\pi}} 
\circ \hat{\pi}^{\ssup {m_\pi}} \circ \bar{\pi},
\end{equation}
where the paths in \eqref{e:concat1} satisfy
\begin{equation}
\label{e:concat2}
\begin{alignedat}{9}
\check{\pi}^{\ssup 1} & \in  \scrP(V, \Pi_{r,A}) 
&\qquad\text{with}\qquad& 
\check{\pi}^{\ssup 1}_i & \notin  \Pi_{r,A}, & \quad\, 0\le i < |\check{\pi}^{\ssup 1}|, 
\\
\hat{\pi}^{\ssup k} & \in  \scrP(\Pi_{r,A}, D_{r,A}^\cc) 
&\qquad\text{with}\qquad& 
\hat{\pi}^{\ssup k}_i & \in  D_{r, A}, & \quad\, 0\le i < |\hat{\pi}^{\ssup k}|, \; 1 \le k \le m_{\pi} - 1, 
\\
\check{\pi}^{\ssup k} & \in  \scrP(D_{r,A}^\cc, \Pi_{r,A}) 
&\qquad\text{with}\qquad& 
\check{\pi}^{\ssup k}_i & \notin  \Pi_{r,A}, & \quad\, 0\le i < |\check{\pi}^{\ssup k}|, \; 2 \le k \le m_\pi, 
\\
\hat{\pi}^{\ssup {m_\pi}} & \in  \scrP(\Pi_{r,A}, V) 
&\qquad\text{with}\qquad& 
\hat{\pi}^{\ssup {m_\pi}}_i & \in  D_{r,A}, & \quad\, 0\le i < |\hat{\pi}^{\ssup {m_\pi}}|, 
\end{alignedat}
\end{equation}
while
\begin{equation}
\label{e:concat3}
\begin{array}{ll} 
\bar{\pi} \in \scrP(D_{r,A}^\cc, V) \text{ and } \bar{\pi}_i \notin \Pi_{r,A} \; \forall\, i \ge 0 
& \text{ if } \hat{\pi}^{\ssup {m_\pi}} \in \scrP(\Pi_{r,A}, D^\cc_{r, A}), \\
\bar{\pi}_0 \in D_{r,A}, |\bar{\pi}| = 0  & \text{ otherwise.}
\end{array}
\end{equation}
Note that the decomposition in \eqref{e:concat1}--\eqref{e:concat3} is unique, 
and that the paths $\check{\pi}^{\ssup 1}$, $\hat{\pi}^{\ssup {m_\pi}}$ and $\bar{\pi}$ can have zero length. 
If $\pi$ is contained in $B_r$, then so are all the paths in the decomposition. 

Whenever $\supp(\pi) \cap \Pi_{r,A} \ne \emptyset$ and $\varepsilon > 0$, we define
\begin{align}
\label{e:defnpikpi}
s_\pi := \sum_{i=1}^{m_\pi} |\check{\pi}^{\ssup i}| + |\bar{\pi}|, \qquad
k^{r,\varepsilon}_\pi := \sum_{i=1}^{m_\pi} M^{r,\varepsilon}_{\check{\pi}^{\ssup i}} + M^{r,\varepsilon}_{\bar{\pi}} 
\end{align}
to be the total time spent in exterior excursions, respectively, 
on moderately low points of the potential visited by exterior excursions (without their last point). 

In case $\supp(\pi) \cap \Pi_{r,A} = \emptyset$, we set $m_\pi := 0$, $s_\pi := |\pi|$ and $k^{r,\varepsilon}_\pi := M^{r,\varepsilon}_{\pi}$. Recall from \eqref{e:deflambdaLApi} that, in this case, $\lambda_{r,A}(\pi) = -\infty$. 

We say that $\pi, \pi' \in \scrP$ are \emph{equivalent}, written $\pi' \sim \pi$, if $m_{\pi} = m_{\pi'}$, $\check{\pi}'^{\ssup i}=\check{\pi}^{\ssup i}$ for all $i=1,\ldots,m_{\pi}$, and $\bar{\pi}' = \bar{\pi}$. If $\pi' \sim \pi$, then $s_{\pi'}$, $k^{r, \varepsilon}_{\pi'}$ and $\lambda_{r,A}(\pi')$ are all equal to the counterparts for $\pi$.

To state our key lemma, we define, for $m,s \in \N_0$,
\begin{equation}
\label{e:defPmn}
\scrP^{(m,s)} = \left\{ \pi \in \scrP \colon\, m_\pi = m, s_\pi = s \right\},
\end{equation}
and denote by
\begin{equation}
\label{def_CLA}
C_{r,A}:= \max \{|\CC| \colon\, \CC \in \mathfrak{C}_{r,A}\}
\end{equation}
the maximal size of the islands in $\mathfrak{C}_{r,A}$.

\begin{lemma}{\bf [Mass of an equivalence class]}
\label{l:fixed_class}
For every $A,\varepsilon > 0$ there exist $c>0$ and $r_0 \in \N$ such that, for all $r \ge r_0$, all $m,s \in \N_0$, all $\pi \in \scrP^{(m,s)}$ with $\supp(\pi) \subset B_r$, all $\gamma > \lambda_{r,A}(\pi) \vee (a_{L_r} -A)$ and all $t \ge 0$,
\begin{multline}
\label{e:fixed_class}
\qquad
\E_{\pi_0} \left[ \texte^{\int_0^t (\xi(X_u) - \gamma)\, \textd u}\, \1_{\{\pi(X_{[0,t]}) \sim \pi\}} \right]  
\\\le \left(C_{r,A}^{1/2} \right)^{\1_{\{m>0\}}} \left(1+\frac{ d_{\max} \, C_{r,A}}{\gamma - \lambda_{r,A}(\pi)} \right)^m \left(\frac{q_A}{d_{\max}}\right)^s \texte^{\left(c-\log^{\ssup 3} L_r \right) k^{r,\varepsilon}_{\pi}}.
\end{multline}
\end{lemma}

\begin{proof}
Fix $A, \varepsilon>0$ and let $c>0$, $n_0 \in \N$ be as given by Lemma~\ref{l:mass_in}. Set
\begin{equation}
I_a^b := \texte^{\int_a^b(\xi(X_u) - \gamma ) \textd u}, \qquad
0 \le a \le b <\infty.
\end{equation} 
We use induction on $m$. Suppose that $m=1$, let $\ell := |\check{\pi}^{\ssup 1}|$. There are two possibilities: either $\bar{\pi}_0$ belongs to $D_{r, A}$ or not. 
First we consider the case $\bar{\pi}_0\in D_{r,A}$, which implies that $|\bar{\pi}|=0$. By the strong Markov property,
\begin{align}
\label{e:fixedclass1}
\E_{\pi_0} \Bigl[I_0^t &\1_{\{ \pi(X_{[0,t]}) \sim \pi \}} \Bigr]
\leq \E_{\pi_0} \left[I_0^{T_\ell} I_{T_\ell}^t  \1_{\{ \pi^{\ssup \ell}(X) 
= \check{\pi}^{\ssup 1}\}}\1_{\{T_\ell < t\}}\1_{\{X_{u+T_\ell} \in D_{r,A} \, 
\forall u\in[0, t - T_\ell]\}} \right] \nonumber\\
&= \E_{\pi_0} \left[I_0^{T_\ell} \1_{\{ \pi^{\ssup \ell}(X) = \check{\pi}^{\ssup 1}\}}
\1_{\{ T_\ell < t \}} \left(\E_{\check{\pi}^{\ssup 1}_{\ell}} \left[ I_{0}^{t-u} \1_{\{\tau_{D^\cc_{r,A}} > t-u \}} \right] \right)_{u=T_\ell} \right].
\end{align}
Put $z=\check{\pi}^{\ssup 1}_{\ell}$. Since $z \in \Pi_{r,A}$, we may write $\CC_z$ to denote the island in $\mathfrak{C}_{r,A}$ containing $z$.  Since $\tau_{D^\cc_{r,A}} = \tau_{\CC^\cc_z}$ $\P_z$-a.s., Lemma~\ref{l:bounds_mass} and the hypothesis on $\gamma$ allow us to bound the inner expectation in \eqref{e:fixedclass1} by $|\CC_z|^{1/2}$.  Applying Lemma~\ref{l:mass_in}, we further bound \eqref{e:fixedclass1} by
\begin{equation}
\label{e:fixedclass2}
|\CC_z|^{1/2} \E_{\pi_0} \left[ I_0^{T_\ell} \1_{\{ \pi^{\ssup \ell}(X) = \check{\pi}^{\ssup 1} \}}\right]
\le C_{r,A}^{1/2} \left(\frac{q_A}{d_{\max}}\right)^{\ell} \texte^{\left(c-\log^{\ssup 3} L_r \right) 
M^{r,\varepsilon}_{\check{\pi}^{\ssup 1}}},
\end{equation}
which proves \eqref{e:fixed_class} for $m=1$ and $\bar{\pi}_0\in D_{r,A}$.

Next consider the case $\bar{\pi}_0 \in D^\cc_{r,A}$. Abbreviating $\sigma := \inf\{ u> T_\ell \colon\, X_u \notin D_{r,A} \}$, write 
\begin{align}
\label{e:fixedclass3}
\E_{\pi_0} \left[I_0^t \1_{\{ \pi(X_{[0,t]}) \sim \pi \}} \right]
\le \E_{\pi_0} \left[I_0^{\sigma} \1_{\{ \pi^{\ssup \ell}(X) = \check{\pi}^{\ssup 1}, \, \sigma < t \}} 
\Big( \E_{\bar{\pi}_0} \left[I_0^{t-u} \1_{\{\pi(X_{[0,t-u]}) = \bar{\pi}\}} \right] \Big)_{u=\sigma} \right].
\end{align}
Let $\ell_* := |\bar{\pi}|$ and note that, since $\bar{\pi}_{\ell_*} \notin \Pi_{r,A}$, by the hypothesis on $\gamma$ we have
\begin{align}
\label{e:fixedclass4}
\E_{\bar{\pi}_0} \left[I_0^{t-u} \1_{\{\pi(X_{[0,t-u]}) = \bar{\pi}\}} \right]
\le \E_{\bar{\pi}_0} \left[I_0^{T_{\ell_*}}\1_{\{\pi^{\ssup{\ell_*}}(X) = \bar{\pi}\}} \right]
\le \left(\frac{q_A}{d_{\max}}\right)^{\ell_*} \texte^{\left(c-\log^{\ssup 3}L_r \right) M^{r,\varepsilon}_{\bar{\pi}}}
\end{align}
where the second inequality holds by Lemma~\ref{l:mass_in}.
On the other hand, by Lemmas~\ref{l:mass_out} and~\ref{l:mass_in},
\begin{align}
\label{e:fixedclass5}
\E_{\pi_0} \left[I_0^{\sigma} \1_{\{ \pi^{\ssup \ell}(X) = \check{\pi}^{\ssup 1}\}} \right]
& = \E_{\pi_0} \left[I_0^{T_\ell} \1_{\{ \pi^{\ssup \ell}(X) = \check{\pi}^{\ssup 1} \}} \right] 
\E_z \left[I_0^{\tau_{\CC^\cc_z}} \right]\nonumber\\
& \le \left( 1 + \frac{d_{\max} \, C_{r,A}}{\gamma  - \lambda_{r,A}(\pi)} \right) \left(\frac{q_A}{d_{\max}}\right)^{\ell} 
\texte^{\left(c-\log^{\ssup 3} L_r \right) M^{r,\varepsilon}_{\check{\pi}^{\ssup 1}}}.
\end{align}
Putting together~\eqref{e:fixedclass3}--\eqref{e:fixedclass5}, we complete the proof of the case $m=1$.
The case $m=0$ follows from \eqref{e:fixedclass4} after we replace $\bar{\pi}$ by $\pi$ and $t-u$ by $t$.

Suppose now that the claim is proved for some $m \ge 1$, and let $\pi \in \scrP^{(m+1,s)}$. Define $\pi' := \check{\pi}^{\ssup 2} \circ \hat{\pi}^{\ssup 2} \circ \cdots \circ \check{\pi}^{\ssup {m+1}} \circ \hat{\pi}^{\ssup {m+1}} \circ \bar{\pi}$. Then $\pi' \in \scrP^{(m,s')}$, where $s = s'+|\check{\pi}^{\ssup 1}|$ and $k^{r,\varepsilon}_\pi = M^{r,\varepsilon}_{\check{\pi}^{\ssup 1}}+k^{r,\varepsilon}_{\pi'}$. Setting $\ell := |\check{\pi}^{\ssup 1}|$, $\sigma := \inf\{ u> T_\ell \colon\; X_u \notin D_{r,A}\}$ and $x:=\check{\pi}^{\ssup 2}_0$, we get 
\begin{equation}
\label{e:fixedclass6}
\E_{\pi_0} \left[ I_0^t \1_{\{\pi(X_{0,t}) \sim \pi \}}\right]
\le \E_{\pi_0} \left[I_0^\sigma \1_{\{\pi^{\ssup \ell}(X) 
= \check{\pi}^{\ssup 1},\, \sigma < t\}} \Big(\E_x \left[ I_0^{t-u} \1_{\{\pi(X_{0,t-u}) 
\sim \pi'\}}\right] \Big)_{s=\sigma} \right],
\end{equation}
from which \eqref{e:fixed_class} follows via the induction hypothesis and \eqref{e:fixedclass5}. 
\end{proof}

%%%

\subsection{Proof of Proposition~\ref{p:massclass}}
\label{ss:proofProp}

\begin{proof}
The proof is based on Lemma~\ref{l:fixed_class}. First define
\begin{equation}
\label{e:defc0A0}
c_0 := 1 + 3 \log \log d_{\max}, \qquad A_0 := d_{\max} \left( \ee^{3 c_0} - 1\right).
\end{equation}
Fix $A>A_0$, $\beta < \alpha$ and $\varepsilon \in (0,\beta/2)$ as in Lemma~\ref{l:bound_mediumpoints}. Let $r_0 \in \N$ be as given by Lemma~\ref{l:fixed_class}, and take $r \ge r_0$ so large that the conclusions of Lemmas~\ref{l:size_comps}--\ref{l:bound_mediumpoints} hold, i.e., assume that the events $\BB_r$ from both lemmas do not occur with either $G = (V,E,\cO) \in \mathfrak{G}_r$ or $G \in \mathfrak{G}_\infty$ accordingly. Fix $x \in B_r$. Recall the definitions of $C_{r,A}$ and $\scrP^{(m,s)}$. Noting that the relation $\sim$ defined below \eqref{e:defnpikpi} is an equivalence relation in $\scrP^{(m,s)}$, we define
\begin{equation}
\label{e:propmassclass2}
\widetilde{\scrP}^{(m,s)}_x := \big\{\text{equivalence classes of the paths in } \scrP(x,V) \cap \scrP^{(m,s)}\big\}.
\end{equation}

\begin{lemma}{\bf [Bound equivalence classes]}
\label{l:propmassclass3}
$|\widetilde{\scrP}^{(m,s)}_x| \le [2 d_{\max} C_{r,A}]^m d_{\max}^s$ for all $m,s \in \N_0$.
\end{lemma}

\begin{proof}
The estimate is clear when $m=0$. To prove that it holds for $m \ge 1$, 
write $\partial \Lambda := \{z \notin \Lambda \colon\, \dist_{G}(z, \Lambda)=1\}$ for $\Lambda \subset V$. 
Then $|\partial \CC \cup \CC| \leq (d_{\max}+1) |\CC| \leq 2d_{\max} C_{r,A}$. 
We define a map 
$\Phi\colon\widetilde{\scrP}^{(m,s)}_x \to\scrP_s(x,V) \times \{1, \ldots, 2d_{\max} C_{r,A} \}^m$ 
as follows. 
For each $\Lambda \subset V$ with $1 \le |\Lambda| \le 2 d_{\max} C_{r,A}$, 
fix an injection $f_\Lambda\colon \Lambda \to \{1, \ldots, 2 d_{\max} C_{r,A} \}$. 
Given a path $\pi \in \scrP^{(m,s)} \cap \scrP(x,V)$, decompose $\pi$ as in \eqref{e:concat1}, 
and denote by $\widetilde{\pi} \in \scrP_s(x, V)$ the shifted concatenation (cf.\ Definition~\ref{def:concat}) 
of $\check\pi^{\ssup 1}, \ldots, \check\pi^{\ssup m}$, $\bar{\pi}$. 
Note that, for $2\le k\le m$, the point $\check\pi^{\ssup k}_0$ lies in $\partial\CC_k$ for some $\CC_k\in \mathfrak{C}_{r,A}$, 
while $\bar{\pi}_0 \in \partial \overline{\CC} \cup \overline{\CC}$ for some $\overline{\CC} \in \mathfrak{C}_{r,A}$. 
Thus, we may set 
\begin{equation} 
\Phi(\pi):= 
\bigl(\widetilde \pi,f_{\partial \CC_2}(\check{\pi}^{\ssup 2}_0),\dots,
f_{\partial \CC_m}(\check{\pi}^{\ssup{m}}_0), f_{\partial \bar{\CC} \cup \bar{\CC}}(\bar{\pi}_0) \bigr).
\end{equation}
As is readily checked, $\Phi(\pi)$ depends only on the equivalence class of $\pi$ and, when restricted to equivalence classes, $\Phi$ is injective. Hence the claim follows.
\end{proof}

Now take $\NN \subset \scrP(x, V)$ as in the statement, and set
\begin{equation}
\label{e:propmassclass1}
\widetilde{\mathcal{N}}^{(m,s)} := \big\{\text{equivalence classes of paths in } 
\NN \cap \scrP^{(m,s)}\big\} \subset \widetilde{\scrP}^{(m,s)}_x.
\end{equation}
For each $\MM \in \widetilde{\NN}^{(m,s)}$, choose a representative $\pi_\MM \in \MM$, and use Lemma~\eqref{l:propmassclass3} to write 
\begin{align}
\label{e:propmassclass6}
& \E_x \left[ \texte^{\int_0^t \xi(X_u) \textd u} \1_{\{\pi(X_{[0,t]}) \in \mathcal{N}\}} \right] 
= \sum_{m, s \in \N_0}  \sum_{\MM \in \widetilde{\mathcal{N}}^{(m,s)}}
\E_x \left[ \texte^{\int_0^t \xi(X_u) \textd u} \1_{\{\pi(X_{[0,t]}) \sim \pi_\MM \}} \right] \nonumber\\
& \quad\qquad \le \sum_{m, s \in \N_0} (2 d_{\max} C_{r,A})^m d_{\max}^s 
\sup_{\pi \in \NN^{(m,s)}} \E_x \left[ \texte^{\int_0^t \xi(X_u) \textd u} \1_{\{\pi(X_{[0,t]}) \sim \pi\}} \right],
\end{align}
where we use the convention $\sup \emptyset = 0$. For fixed $\pi \in \mathcal{N}^{(m,s)}$, by \eqref{e:cond_massclass1}, we may apply \eqref{e:fixed_class} and Lemma~\ref{l:size_comps} to obtain, for all $r$ large enough and with $c_0$ as in \eqref{e:defc0A0},
\begin{align}
\label{e:propmassclass7}
(2 d_{\max})^m d_{\max}^s \E_x \left[ \texte^{\int_0^t \xi(X_u) \textd u} \1_{\{\pi(X_{[0,t]}) \sim \pi\}} \right] 
\le \texte^{t \gamma_\pi } \texte^{c_0 m S_r} 
 q_A^s \texte^{\left(c-\log^{\ssup 3}L_r \right) k^{r,\varepsilon}_\pi}.
\end{align}

We next claim that, for $r$ large enough and $\pi \in \NN^{(m,s)}$,
\begin{equation}
\label{e:propmassclass7.1}
s \ge \left[(m-1)\vee 1 \right] S_r .
\end{equation}
Indeed, when $m\ge 2$, $|\supp(\check{\pi}^{\ssup i})| \ge S_r$ for all $2 \le i \le m$. When $m=0$, $|\supp(\pi)| \ge \max_{1 \le \ell \le |\pi|} |\pi_\ell -x| \ge (\log L_r)^\kappa \gg S_r$ by assumption. When $m=1$, the latter assumption and Lemma~\ref{l:size_comps} together imply that $\supp(\pi) \cap D^\cc_{r,A} \neq \emptyset$, and so either $|\supp(\check{\pi}^{\ssup 1})| \geq S_r$ or $|\supp(\check{\pi}^{\ssup 1})|\ge S_r$. Thus, \eqref{e:propmassclass7.1} holds by \eqref{e:defnpikpi} and \eqref{e:def_Sr}.

Note that $q_A < \ee^{-3c_0}$, so
\begin{equation}
\label{e:propmassclass7.5}
\sum_{m \geq 0} \sum_{s \geq [(m-1)\vee 1] S_r} \ee^{c_0 m S_r} q_A^s
= \frac{q_A^{S_r} + \ee^{c_0 S_r}q_A^{S_r} + \sum_{m \geq 2} \ee^{c_0 S_r m } q_A^{(m-1)S_r}}{1-q_A}
\leq \frac{4 \ee^{-c_0 S_r}}{1-q_A} < 1
\end{equation}
for $r$ large enough. Inserting this back into~\eqref{e:propmassclass6}, we obtain
\begin{equation}
\label{e:intermediatemassclass}
\log \E_x \left[ \texte^{\int_0^t \xi(X_s) \textd s} \1_{\{\pi(X_{0,t}) \in \mathcal{N}\}} \right]
\le \sup_{\pi \in \mathcal{N}} \Big\{ t \gamma_\pi + \left(c- \log^{\ssup 3}L_r \right) 
k^{r,\varepsilon}_\pi \Big\}.
\end{equation}

Thus the proof will be finished once we show that, for some $\varepsilon' > 0$, whp (respectively, almost surely eventually) as $n \to \infty$, all $\pi \in \NN$ satisfy
\begin{equation}
\label{e:propmassclass9}
k^{r,\varepsilon}_\pi \ge \dist_{G}(x,z_{\pi})(1-2(\log L_r)^{-\varepsilon'}).
\end{equation}
To that end, we define for each $\pi \in \NN$ an auxiliary path $\pi_\star$ as follows. 
First note that by using our assumptions we can find points $z', z'' \in \supp(\pi)$ (not necessarily distinct) such that 
\begin{equation}
\label{e:propmassclass10}
\dist_{G}(x,z') \geq (\log L_r)^\kappa, \qquad \dist_{G}(z'', z_\pi) \leq 2 M_A S_r,
\end{equation}
where the latter holds by Lemma~\ref{l:size_comps}. 
Write  $\{z_1, z_2 \} = \{z', z''\}$ with $z_1$, $z_2$ ordered according to their hitting times by $\pi$, 
i.e., $\inf\{ \ell \colon \pi_\ell = z_1 \} \leq \inf\{\ell \colon \pi_\ell = z_2\}$. 
Define $\pi_e$ as the concatenation of the loop erasure of $\pi$ between $x$ and $z_1$ and the loop erasure of $\pi$ between $z_1$ and $z_2$. 
Since $\pi_e$ is the concatenation of two self-avoiding paths, it visits each point at most twice. 
Finally, define $\pi_\star \sim \pi_e$ by substituting the excursions of $\pi_e$ from $\Pi_{r,A}$ to $D_{r,A}^\cc$ 
by direct paths between the corresponding endpoints, i.e., 
substitute each $\hat{\pi}_e^{\ssup i}$ with $|\hat{\pi}_e^{\ssup i}|=\ell_i$, $(\hat{\pi}_e^{\ssup i})_0 = x_i \in \Pi_{r,A}$ 
and $(\hat{\pi}_e^{\ssup i})_{\ell_i} = y_i \in D_{r,A}^\cc$ 
by a shortest-distance path $\widetilde{\pi}_\star^{\ssup i}$ 
with the same endpoints and $|\widetilde{\pi}_\star^{\ssup i}| = \dist_{G}(x_i, y_i)$.
Since $\pi_\star$ visits each $x \in \Pi_{r,A}$ at most $2$ times,
\begin{equation}
\label{e:propmassclass11}
\begin{aligned}
k^{r,\varepsilon}_\pi \ge k^{r,\varepsilon}_{\pi_\star} \geq M^{r,\varepsilon}_{\pi_\star} - 2 |\supp(\pi_\star)\cap \Pi_{r,A}|(S_r+1) \geq M^{r,\varepsilon}_{\pi_\star} - 4 |\supp(\pi_\star)\cap \Pi_{r,A}|  S_r.
\end{aligned}
\end{equation}
Note that $M_{\pi_\star}^{r, \varepsilon} \geq \left|\{x \in \supp(\pi_\star) \colon\, \xi(x) \leq (1-\varepsilon) a_{L_r}\} \right| - 1$
and, by \eqref{e:propmassclass10}, $|\supp(\pi_\star)| \geq \dist_{G}(x,z') \geq (\log L_r)^\kappa \gg (\log L_r)^{\alpha+2\varepsilon'}$ for some $0<\varepsilon'<\varepsilon$. 
Applying Lemmas~\ref{l:bound_mediumpoints}--\ref{l:boundhighexceedances} and using \eqref{e:def_Sr} and $L_r > r$, we obtain, for $r$ large enough,
\begin{equation}
\label{e:propmassclass12}
\begin{aligned}
k^{r,\varepsilon}_\pi 
& \geq |\supp(\pi_\star)|\left( 1 - \frac{2}{(\log L_r)^{\varepsilon}} 
-  \frac{4 S_r}{(\log L_r)^{\alpha+2\varepsilon'}}\right) 
\geq |\supp(\pi_\star)|\left( 1 - \frac{1}{(\log L_r)^{\varepsilon'}}\right).
\end{aligned}
\end{equation}
On the other hand, since $|\supp(\pi_\star)| \geq (\log L_r)^\kappa$ and by \eqref{e:propmassclass10} again,
\begin{equation}
\label{e:propmassclass13}
\begin{aligned}
\left|\supp(\pi_\star) \right| &= \big(\left|\supp(\pi_\star) \right| + 2 M_A S_r\big) - 2 M_A S_r\\
& \geq \left( \dist_{G}(x,z'') + 2 M_A S_r \right) \left( 1-\frac{2 M_A S_r}{(\log L_r)^\kappa} \right) \\
& \geq \dist_{G}(x,z_\pi)\left( 1-\frac{1}{(\log L_r)^{\varepsilon'}} \right).
\end{aligned}
\end{equation}
Now \eqref{e:propmassclass9} follows from \eqref{e:propmassclass12}--\eqref{e:propmassclass13}.
\end{proof}

%%%%%%%%%%%%% SECTION 4 %%%%%%%%%%%%%%%%%%%%%%%%%%%%%%%%%%

\section{Proof of Theorem~\ref{t:QLyapGWT}} 
\label{s:proofoutline}

This section is devoted to the proof of Theorem \ref{t:QLyapGWT}. We note that, after replacing $d_{\max}$ by $d_{\max} \vee D_0$ if necessary, we may assume without loss of generality that
\begin{equation}
\label{e:GWinfrakG}
\GW \in \mathfrak{G}^{(\vartheta)}_\infty.
\end{equation}
%

%%%%%%%%%%%%%%%%%%%%%%%%

\subsection{Lower bound}

In this section we give the proof of the lower bound for the large-$t$ asymptotics of the total mass. This proof already explains the random mechanism that produces the main contribution to the total mass. This mechanism comes from an \emph{optimization} of the behavior of the random path in the Feynman-Kac formula, which in turn comes from the existence of a \emph{favorite region} in the random graph, both in terms of the local graph structure and the high values of the potential in this local graph structure. The optimality is expressed in terms of a distance to the starting point $\cO$ that can be reached in a time $o(t)$ with a sufficiently high probability, such that time $t-o(t)$ is left for staying inside the favorite region, thus yielding a maximal contribution to the Feynman-Kac formula. The latter is measured in terms of the \emph{local eigenvalue} of the Anderson operator $\Delta+\xi$, which in turn comes from high values and optimal shape of the potential $\xi$ in the local region.

We write the total mass of the solution of \eqref{e:PAM} in terms of the Feynman-Kac formula as
\begin{equation}
\label{FKform}
U(t)=\E_\cO\Big[\exp\Big\{\int_0^t \xi(X_s)\,\dd s\Big\}\Big],
\end{equation}
where $(X_s)_{s \geq 0}$ is the continuous-time random walk on $\GW$, i.e., the Markov chain with generator $\Delta_{\GW}=\Delta$, the Laplacian on $\GW$, starting from the origin $\cO$. As usual in the literature of the PAM, this formula is the main point of departure for our proof.

Fix $\varepsilon>0$. By the definition of $\widetilde{\chi}$, there exists an infinite rooted tree $T=(V',E',\YY)$ with degrees in $\supp(D_g)$ such that $\chi_T(\varrho) < \widetilde{\chi}(\varrho) + \tfrac14 \varepsilon$. Let $Q_r = B^T_r(\YY)$ be the ball of radius $r$ around $\YY$ in $T$. By Proposition~\ref{p:dualrepchi} and \eqref{e:defhatchi}, there exist a radius $R \in \N$ and a potential profile $q\colon B^T_R \to\R$ with $\cL_{Q_R}(q;\varrho)<1$ (in particular, $q\leq 0$) such that
\begin{equation}
\label{e:prLB0}
\lambda_{Q_R}(q;T) \geq -\widehat{\chi}_{Q_R}(\varrho;T) - \tfrac12 \varepsilon 
> -\widetilde{\chi}(\varrho) - \varepsilon.
\end{equation}

For $\ell\in\N$, let $B_\ell = B_\ell(\cO)$ denote the ball of radius $\ell$ around $\cO$ in $\GW$. We will show next that, 
almost surely eventually as $\ell \to \infty$, $B_\ell$ contains a copy of the ball $Q_R$ where $\xi$ is lower bounded by $\varrho\log\log |B_\ell| + q$.

\begin{proposition}{\bf [Balls with high exceedances]}
\label{p:existencesubtree}
$\Probgr \times \Prob$-almost surely eventually as $\ell \to \infty$, there exists a vertex $z \in B_\ell$ with $B_{R+1}(z) \subset B_\ell$ and an isomorphism $\varphi:B_{R+1}(z) \to Q_{R+1}$ such that $\xi \geq \varrho \log \log |B_\ell| + q \circ \varphi$ in $B_R(z)$. In particular,
\begin{equation}
\lambda_{B_R(z)}(\xi; \GW) > \varrho \log \log |B_\ell| - \widetilde{\chi}(\varrho) - \varepsilon.
\end{equation}
Any such $z$ necessarily satisfies $|z| \geq c \ell$ $\Probgr \times \Prob$-almost surely eventually as $\ell \to \infty$ 
for some constant $c = c(\varrho, \vartheta, \widetilde{\chi}(\varrho), \varepsilon) >0$.
\end{proposition}

\begin{proof}
First note that, as a consequence of the definition of $\GW$, it may be shown straightforwardly that, for some $p=p(T, R) \in (0,1)$ and $\Probgr$-almost surely eventually as $\ell \to \infty$, there exist $N \in \N$, $N \geq p |B_\ell|$ and distinct $z_1, \ldots, z_N \in B_\ell$ such that $B_{R+1}(z_i) \cap B_{R+1}(z_j) = \emptyset$ for $1\leq i\neq j\leq N$ and, for each $1 \leq i \leq N$, $B_{R+1}(z_i) \subset B_\ell$  and $B_{R+1}(z_i)$ is isomorphic to $Q_{R+1}$. Now, by \eqref{e:DE}, for each $i \in \{1,\ldots, N\}$,
\begin{equation}
\Prob\big(\xi\geq \varrho\log\log|B_\ell|+q \mbox{ in }B_R(z_i)\big)=|B_\ell|^{-\cL_{Q_R}(q)}.
\end{equation}
Using additionally that $|B_\ell| \geq \ell$ and $1-x \leq \ee^{-x}$, $x\in\R$, we obtain
\begin{equation*}
\begin{aligned}
\Prob(\not \exists i \in \{1,\ldots,N\} \colon \xi\geq \varrho\log\log|B_\ell|+q \mbox{ in }B_R(z_i))
= \left( 1 - |B_\ell|^{-\cL_{Q_R}(q)} \right)^N 
\leq \ee^{-p \ell^{1-\cL_{Q_R}(q)}}
\end{aligned}
\end{equation*}
which is summable in $\ell \in \N$, so the proof of the first statement is completed using the Borel-Cantelli lemma. As for the last statement, note that, by \eqref{e:monot_princev}, Lemma~\ref{l:maxpotential} and $L_r \sim \vartheta r$,
\begin{equation}
\lambda_{B_{c\ell}}(\xi;\GW) \leq \max_{x \in B_{c \ell}} \xi(x) < a_{L_{c \ell}} + o(1) 
< a_{L_\ell} + \varrho \log c \vartheta + o(1) < a_{L_\ell} - \widetilde{\chi}(\varrho)-\varepsilon
\end{equation}
provided $c>0$ is small enough.
\end{proof}

\begin{proof}[Proof of the lower bound in \eqref{e:QLyapGWT}]
Let $z$ be as in Proposition~\ref{p:existencesubtree}. 
Write $\tau_z$ for the hitting time of $z$ by the random walk $X$.
For any $s\in (0,t)$, we obtain a lower bound for $U(t)$ as follows:
\begin{equation}
\label{lowbound1}
\begin{aligned}
U(t)&\geq \E_\cO\Big[\exp\Big\{\int_0^t \xi(X_u)\,\dd u\Big\}\,\1_{{\{\tau_z\leq s\}}}\,
\1_{{\{X_u\in B_{R}(z)\,\forall u\in[\tau_z,t]\}}}\Big]\\
&=\E_\cO\Big[\ee^{\int_0^{\tau_z} \xi(X_u)\,\dd u}\,\1_{{\{\tau_z\leq s\}}}\,
\E_z\Big[\ee^{\int_0^{v} \xi(X_u)\,\dd u}\,\1_{{\{X_u\in T\,\forall u\in [0,v]\}}}\Big]\Big|_{v=t-\tau_z}\Big],
\end{aligned}
\end{equation}
where we use the strong Markov property at time $\tau_z$. We first bound the last term in the integrand in \eqref{lowbound1}. Since $\xi \geq \varrho \log\log |B_\ell| +q $ in $B_R(z)$, 
\begin{equation}
\begin{aligned}
\E_z\Big[\ee^{\int_0^{v} \xi(X_u)\,\dd u} \1_{\{X_u\in B_R(z)\,\forall u\in [0,v]\}}\Big]
& \geq \ee^{v \varrho \log \log |B_\ell|} \E_{\YY}\Big[\ee^{\int_0^{v} q(X_u)\,\dd u} 
\1_{\{X_u\in Q_R\,\forall u\in [0,v]\}}\Big] \\
& \geq \e^{ v \varrho \log \log |B_\ell|} \ee^{v \lambda_{Q_R}(q;T)} \phi^{\ssup 1}_{Q_R}(\YY)^2 \\
& > \exp \big\{ v \left(\varrho \log\log |B_{\ell}| -  \widetilde{\chi}(\varrho) - \varepsilon \right) \big\},
\end{aligned}
\end{equation}
for large $v$, where we used that $B_{R+1}(z)$ is isomorphic to $Q_{R+1}$ and applied Lemma~\ref{l:bounds_mass} and \eqref{e:prLB0}.
On the other hand, since $\xi\geq0$,  
\begin{equation}
\E_\cO \Big[\exp\Big\{\int_0^{\tau_z} \xi(X_u)\,\dd u\Big\}\1{\{\tau_z\leq s\}}\Big]
\geq \P_\cO(\tau_z\leq s),
\end{equation}
and we can bound the latter probability from below by the probability that the random walk runs along a shortest path from the root $\cO$ to $z$ within a time at most $s$. Such a path $(y_i)_{i=0}^{|z|}$ has $y_0 = \cO$, $y_{|z|} = z$, $y_i \sim y_{i-1}$ for $i=1, \ldots, |z|$, has at each step from $y_i$ precisely $\deg(y_i)$ choices for the next step with equal probability, and the step is carried out after an exponential time $E_i$ with parameter $\deg(y_i)$. This gives
\begin{equation}
\begin{aligned}
\Prob_\cO(\tau_z\leq s) 
& \geq \Big(\prod_{i=1}^{|z|}\frac 1{\deg(y_i)}\Big) P\Big(\sum_{i=1}^{|z|} E_i \leq s\Big) 
\geq  d_{\max}^{-|z|}
{\rm Poi}_{d_{\rm min} s}([|z|,\infty)),
\end{aligned}
\end{equation}
where ${\rm Poi}_\gamma$ is the Poisson distribution with parameter $\gamma$, and $P$ is the generic symbol for probability. Summarising, we obtain
\begin{equation}
\label{lowbound2}
\begin{aligned}
U(t) 
& \geq d_{\max}^{-|z|} \e^{-d_{\rm min} s}\frac{(d_{\rm min} s)^{|z|}}{|z|!}
\e^{(t-s)\left[\varrho\log\log |B_{\ell}| - \widetilde{\chi}(\varrho) - \varepsilon \right]} \\
& \geq \exp \left\{-d_{\min} s + (t-s)\left[\varrho\log\log |B_{\ell}| - \widetilde{\chi}(\varrho) 
- \varepsilon \right] - |z| \log \left( \frac{d_{\max}}{d_{\min}}\frac{|z|}{s}\right) \right\} \\
& \geq \exp \left\{-d_{\min} s + (t-s)\left[\varrho\log\log |B_{\ell}| - \widetilde{\chi}(\varrho) 
- \varepsilon \right] - \ell \log \left( \frac{d_{\max}}{d_{\min}}\frac{\ell}{s}\right) \right\},
\end{aligned}
\end{equation}
where for the last inequality we assume $s \leq |z|$ and use $\ell \geq |z|$. Further assuming that $\ell = o(t)$, we see that the optimum over $s$ is obtained at 
\begin{equation}
s= \frac{\ell}{d_{\min}+\varrho\log\log|B_{\ell}|-\widetilde{\chi}(\varrho) - \varepsilon} =o(t).
\end{equation} 
Note that, by Proposition~\ref{p:existencesubtree}, this $s$ indeed satisfies $s\leq |z|$. Applying \eqref{e:volumerateGW} we get, after a straightforward computation, almost surely eventually as $t \to \infty$,
\begin{equation}
\label{Ulowboundwithr}
\frac 1t\log U(t)\geq \varrho\log \log |B_\ell| -\frac{\ell}t \log\log \ell - \widetilde{\chi}(\varrho) 
- \varepsilon + O\left( \frac{\ell}{t} \right).
\end{equation}
Analysing the main terms above and using $\log |B_\ell| \sim \vartheta \ell$, we find that the optimal $\ell$ satisfies $\ell \log\log \ell -\frac{\ell}{\log \ell} \sim t\varrho$, i.e., $\ell\sim\varrho t/\log\log t = \mathfrak{r}_t$. For this choice we obtain
\begin{equation}
\label{UlowboundGWT1}
\frac 1t\log U(t)\geq \varrho\log \log |B_{\mathfrak{r}_t}| - \mathfrak{r}_t \log \log \mathfrak{r}_t
-\widetilde{\chi}(\varrho) - \varepsilon +O\left( \frac{1}{\log \log t}\right).
\end{equation}
Substituting $\log |B_r| \sim \vartheta r$ and the definition of $\mathfrak{r}_t$, we obtain, $\Probgr \times \Prob$-almost surely,
\begin{equation}
\label{UlowboundGWT2}
\liminf_{t \to \infty}
\left\{ \frac 1t\log U(t)- \varrho\log \left(\frac{\varrho \vartheta t}{\log \log t}\right) \right\}
\geq -\varrho -\widetilde{\chi}(\varrho) - \varepsilon.
\end{equation}
Since $\varepsilon>0$ is arbitrary, the proof of the lower bound in \eqref{e:QLyapGWT} is complete.
\end{proof}

%%%

\subsection{Upper bound}

In this section we prove the upper bound in \eqref{e:QLyapGWT}. A first step is to reduce the problem to a ball of radius $t \log t$. Here we include more general graphs.

\begin{lemma}{\bf [Spatial truncation]}
\label{l:firsttruncation}
For any $c>0$ and any $\ell_t \in \N$, $\ell_t \geq c t \log t$, 
\begin{equation}
\label{e:FT1}
\sup_{G \in \mathfrak{G}_{\ell_t}} \E_{\cO} \left[\ee^{\int_0^t \xi(X_s) \dd s} 
\1_{\{\tau_{B^\cc_{\ell_t}}< t\}} \right] \leq \ee^{-\ell_t} \quad \text{ whp as } t \to \infty.
\end{equation}
Moreover, for any $G \in \mathfrak{G}^{(\vartheta)}_\infty$,
\begin{equation}
\label{e:FT2}
\E_{\cO} \left[\ee^{\int_0^t \xi(X_s) \dd s} \1_{\{\tau_{B^\cc_{\ell_t}}< t\}} \right] \leq  
\ee^{-\ell_t} \quad \text{ $\Prob$-a.s. eventually as } t \to \infty.
\end{equation}
\end{lemma}

\begin{proof}
For $r \geq \ell_t$ and $G \in \mathfrak{G}_{\ell_t}$, let
\begin{equation}
\BB_r := \left\{ \max_{x \in B_r} \xi(x) \geq a_{L_r} + 2 \varrho\right\}.
\end{equation}
By Lemma~\ref{l:maxpotential} and a union bound, we see that
\begin{equation}
\label{e:prFT1}
\sup_{G \in \mathfrak{G}_{\ell_t}} \Prob \left(\bigcup_{r \geq \ell_t} \BB_r \right) 
\leq \sum_{r \geq \ell_t} \sup_{G \in \mathfrak{G}_{\ell_t}} \Prob (\BB_r) 
\underset{t \to \infty}{\longrightarrow} 0,
\end{equation}
while, for $G \in \mathfrak{G}^{(\vartheta)}_\infty$, by the Borel-Cantelli lemma,
\begin{equation}
\label{e:prFT2}
\bigcup_{r \geq \ell_t} \BB_r
\text{ does not occur $\Prob$-a.s. eventually as } t\to \infty.
\end{equation}
We may therefore work on the event $\bigcap_{r \geq \ell_t} \BB_r^\cc$. On this event, we may write
\begin{align}
\label{e:prFT3}
\E_{\cO} \left[\ee^{\int_0^t \xi(X_s) \dd s} \1_{\{\tau_{B^\cc_{\ell_t}}< t\}} \right]
& = \sum_{r \geq \ell_t} \E_{\cO} \left[\ee^{\int_0^t \xi(X_s) \dd s} 
\1_{\{\sup_{s \in [0,t]}|X_s| = r \}} \right] \nonumber\\
& \leq  \ee^{C t} \sum_{r \geq \ell_t}\, \ee^{\varrho t \log r } \, 
\P_{\cO} \left( J_t \geq r \right),
\end{align}
where $J_t$ is the number of jumps of $X$ up to time $t$, $C = \varrho (2 + \log \log d_{\max})$, and we use that $|B_r| \leq d_{\max}^r$. 
Note that $J_t$ is stochastically dominated by a Poisson random variable with parameter $t d_{\max}$. Hence
\begin{equation}
\P_{\cO} \left( J_t \geq r \right) \leq \frac{(t d_{\max})^r}{r!} \leq 
\exp \left\{-r \log\left( \frac{r}{\ee t d_{\max}}\right) \right\}
\end{equation}
for large $r$. Using $\ell_t \geq c t \log t$, we can check that, for $r \geq \ell_t$ and $t$ large enough,
\begin{equation}
r \log\left( \frac{r}{\ee t d_{\max}}\right) - \varrho t \log r > 2 r
\end{equation}
and thus \eqref{e:prFT3} is at most $\ee^{-\ell_t} \ee^{-\ell_t + Ct + 2} < \ee^{-\ell_t}$.
\end{proof}

In order to be able to apply Proposition~\ref{p:massclass} in the following, 
we need to make sure that all paths considered exit a ball with a slowly growing radius. 

\begin{lemma}{\bf [No short paths]}
\label{l:noshortpaths}
For any $\gamma \in (0,1)$,
\begin{equation}
\label{e:NSP1}
\sup_{G \in \mathfrak{G}_{\lceil t^\gamma \rceil}} \frac{\E_{\cO} \left[\ee^{\int_0^t \xi(X_s) \dd s} 
\1_{\{\tau_{B^\cc_{\lceil t^\gamma \rceil}} > t\}} \right]}{U(t)} = o(1) \quad \text{ whp as } t \to \infty.
\end{equation}
Moreover, for any $G \in \mathfrak{G}_\infty$,
\begin{equation}
\label{e:NSP2}
\lim_{t\ \to \infty} \frac{\E_{\cO} \left[\ee^{\int_0^t \xi(X_s) \dd s} 
\1_{\{\tau_{B^\cc_{\lceil t^\gamma \rceil}} > t\}} \right]}{U(t)} = 0 \quad \text{ $\Prob$-a.s. almost surely}.
\end{equation}
\end{lemma}

\begin{proof}
By Lemma~\ref{l:maxpotential} with $g_r = 2 \varrho \log r$, we may assume that
\begin{equation}
\max_{x \in B_{\lceil t^\gamma \rceil}} \xi(x) \leq \varrho \log \log L_{\lceil t^\gamma \rceil} + 2 \varrho 
= \gamma \varrho \log t + 2 \varrho + o(1) \text{ as } t \to \infty.
\end{equation}
By \eqref{UlowboundGWT2}, for some constant $C>0$,
\begin{equation*}
\frac{\E_{\cO} \left[\ee^{\int_0^t \xi(X_s) \dd s} \1_{\{\tau_{B^\cc_{\lceil t^\gamma \rceil}} > t\}} \right]}{U(t)}
\leq \ee^{C t \log^{\ssup 3} t}\ee^{ - (1-\gamma)\varrho t \log t} \underset{t \to \infty}{\longrightarrow} 0.
\qedhere
\end{equation*}
\end{proof}

For the remainder of the proof we fix $\gamma \in (\alpha,1)$ with $\alpha$ as in \eqref{e:def_Sr}. Let
\begin{equation}
\label{e:defrkt}
K_t := \lceil t^{1-\gamma} \log t \rceil, \qquad r^{(k)}_t := k \lceil t^\gamma \rceil, \; 
1 \leq k \leq K_t \qquad \text{and} \qquad \ell_t := K_t \lceil t^\gamma \rceil \geq t \log t.
\end{equation}
For $1 \leq k \leq K_t$ and $G \in \mathfrak{G}^{(\vartheta)}_\infty$, define
\begin{equation}
\cN^{\ssup k}_t := \left\{ \pi \in \scrP(\cO, V) \colon\, \supp(\pi) \subset B_{r^{\ssup {k+1}}_t}, 
\supp(\pi)\cap B^\cc_{r^{\ssup k}_t} \neq \emptyset \right\}
\end{equation}
and set
\begin{equation}
U^{\ssup k}_t := \E_\cO \left[ \ee^{\int_0^t \xi(X_s) \dd s} \1_{\{\pi_{[0,t]}(X) \in \cN^{\ssup k}_t \}}\right].
\end{equation}
Recall the scale $\mathfrak{r}_t = \varrho t /\log \log t$.

\begin{lemma}{\bf [Upper bound on $U^{\ssup k}_t$]}
\label{l:UBpieces}
For any $\varepsilon>0$ and any $G \in \mathfrak{G}^{(\vartheta)}_\infty$, $\Prob$-almost surely eventually as $t \to \infty$,
\begin{equation}
\label{e:UBpieces}
\sup_{1 \leq k \leq K_t} \frac 1t \log U^{\ssup k}_t \leq \varrho \log (\vartheta \mathfrak{r}_t) 
- \varrho - \widetilde{\chi}(\varrho) + \varepsilon. 
\end{equation}
\end{lemma}

\begin{proof}
Before we apply Proposition~\ref{p:massclass}, we first do a bit of analysis. For $c>0$, let
\begin{equation}
F_{c,t}(r) := \varrho \log (\vartheta r) - \frac{r}{t} \left( \log \log r - c\right), \qquad r>0.
\end{equation}
Note that $F_{c,t}$ is maximized at a point $r_{c,t}$ satisfying
\begin{equation}
\varrho t = r_{c,t} \log \log r_{c,t} -c r_{c,t} + \frac{r_{c,t}}{\log r_{c,t}}.
\end{equation}
In particular, $r_{c,t} \sim \mathfrak{r}_t$, which implies
\begin{equation}
\label{e:prUBpieces1}
\sup_{r > 0} F_{c,t}(r) \leq \varrho \log (\vartheta \mathfrak{r}_t) - \varrho + o(1) \qquad \text{ as } t \to \infty.
\end{equation}

Next, fix $k \in \{1, \ldots, K_t\}$. For $\pi \in \cN^{\ssup k}_t$, let
\begin{equation}
\gamma_\pi := \lambda_{r^{\ssup {k+1}}_t, A}(\pi) + \exp\{-S_{\lceil t^\gamma \rceil} \}, 
\qquad z_\pi \in \supp(\pi), |z_\pi| > r^{\ssup k}_t.
\end{equation}
By Proposition~\ref{p:massclass}, almost surely eventually as $t \to \infty$,
\begin{equation}
\label{e:prUBpieces2}
\begin{aligned}
\frac 1t \log U^{\ssup k}_t 
\leq \gamma_\pi + \frac{|z_\pi|}{t} \left( \log \log r^{(k+1)}_t -c_A +o(1) \right).
\end{aligned}
\end{equation}
Using Corollary~\ref{c:eigislandsGW} and $\log L_r \sim \vartheta r$, we bound
\begin{equation}
\begin{aligned}
\gamma_\pi 
\leq \varrho \log (\vartheta r^{(k+1)}_t) - \widetilde{\chi}(\varrho) + \tfrac12 \varepsilon + o(1).
\end{aligned}
\end{equation}
Moreover, $|z_\pi| > r^{\ssup {k+1}}_t - \lceil t^\gamma \rceil$ and
\begin{equation}
\frac{\lceil t^\gamma \rceil}{t} \left(\log \log r^{\ssup {k+1}}_t - c_A \right) \leq 
\frac{2}{t^{1-\gamma}} \log \log (2 t \log t) = o(1),
\end{equation}
which allows us to further bound \eqref{e:prUBpieces2} by
\begin{equation}
\begin{aligned}
\varrho \log (\vartheta r^{\ssup {k+1}}_t) - \frac{r^{\ssup {k+1}}_t}{t} 
\left( \log \log r^{\ssup {k+1}}_t - 2c_A \right) - \widetilde{\chi}(\varrho) + \tfrac12 \varepsilon + o(1).
\end{aligned}
\end{equation}
Applying \eqref{e:prUBpieces1} we obtain $\displaystyle \frac 1t \log U^{\ssup k}_t  < \varrho \log (\vartheta \mathfrak{r}_t) - \varrho - \widetilde{\chi}(\varrho) + \varepsilon$.
\end{proof}

\begin{proof}[Proof of upper bound in \eqref{e:QLyapGWT}]
To avoid repetition, all statements are assumed to be made $\Probgr \times \Prob$-almost surely eventually as $t \to \infty$. Let $G = \GW$ and note that $\GW \in \mathfrak{G}^{(\vartheta)}_\infty$ almost surely, where $\vartheta$ is as in \eqref{e:volumerateGW}. Define
\begin{equation}
U^{\ssup 0}_t := \E_\cO \left[ \ee^{\int_0^t \xi(X_s) \dd s} 
\1_{\{\tau_{B^\cc_{\lceil t^\gamma \rceil}}>t\}}\right],
\qquad
U^{\ssup \infty}_t := \E_\cO \left[ \ee^{\int_0^t \xi(X_s) \dd s} 
\1_{\{\tau_{B^\cc_{\lceil t \log t \rceil}} \leq t\}}\right].
\end{equation}
Note that
\begin{equation}
U(t) \leq U_t^{\ssup 0} + U^{\ssup \infty}_t + K_t \max_{1 \leq k \leq K_t} U^{\ssup k}_t
\end{equation}
and, since $U^{\ssup 0}_t + U^{\ssup \infty}_t \leq o(1) U(t)$ by Lemmas~\ref{l:firsttruncation}--\ref{l:noshortpaths}
and \eqref{UlowboundGWT1}, 
\begin{equation}
U(t) \leq 2K_t \max_{1 \leq k \leq K_t} U^{\ssup k}_t 
\quad \text{ and so } \quad
\frac 1 t \log U(t) \leq \frac{\log (2 K_t) }{t} + \max_{1 \leq k \leq K_t} \frac 1t \log U^{\ssup k}_t.
\end{equation}
By Lemma~\ref{l:UBpieces} and \eqref{e:defrkt}, for any $\varepsilon>0$,
\begin{equation}
\begin{aligned}
\frac 1 t \log U(t) 
& \leq \varrho \log (\vartheta \mathfrak{r}_t) - \varrho - \widetilde{\chi}(\varrho) + \varepsilon + o(1)\\
\end{aligned}
\end{equation}
therefore, $\Probgr \times \Prob$-almost surely,
\begin{equation}
\limsup_{t \to \infty} \left\{\frac 1t \log U(t) 
-  \varrho \log \left( \frac{\vartheta \varrho t}{\log \log t} \right) \right\} 
\leq - \varrho -\widetilde{\chi}(\varrho) + \varepsilon.
\end{equation}
Since $\varepsilon>0$ is arbitrary, this completes the proof of the lower bound in \eqref{e:QLyapGWT}.
\end{proof}

%%%%%%%%%%%%% SECTION 5 %%%%%%%%%%%%%%%%%%%%%%

\section{Proof of Theorem \ref{t:QLyapCM} }
\label{s:proofThCM}

In this section we give the proof of Theorem~\ref{t:QLyapCM}. The proof is based on the fact that, up to a radius growing slower than $\log \Phi_n$ (cf.\ \eqref{e:defPhin}), the configuration model equals a Galton-Watson tree with high probability. From this the result will follow via Theorem~\ref{t:QLyapGWT} and Lemma~\ref{l:firsttruncation}.

To describe the associated Galton-Watson tree, we define a random variable $D_\star$ as the size-biased version of $D$ in Assumption~(CM)$(1)$, i.e.,
\begin{equation}
\label{e:defDstar}
P(D_\star = k) = \frac{k P(D = k)}{E[D]}.
\end{equation}

\begin{proposition}{\bf [Coupling of $\, \UG_n$ and $\GW$]}
\label{p:couplingCMGW}
Let $\UG_n=(V_n, E_n, \cO_n)$ be the uniform simple random graph with degree sequence $\mathfrak{d}^{\ssup n}$ satisfying Assumption~(CM), and let $\GW = (V,E, \cO)$ be a Galton-Watson tree with initial degree distribution $D_0=D$ and general degree distribution $D_g = D_\star$. There exists a coupling $\widetilde{\Prob}$ of $\,\UG_n$ and $\GW$ such that, for any $m_n \in \N$ satisfying $1 \ll m_n \ll \log \Phi_n$,
\begin{equation}
\lim_{n \to \infty} \widetilde{\Prob} \left( B^{\UG_n}_{m_n}(\cO_n) = B^{\GW}_{m_n}(\cO) \right) = 1.
\end{equation}
\end{proposition}

\begin{proof}
For $\CM_n$ in place of $\UG_n$, this is a consequence of the proof of \cite[Proposition~5.4]{vdH17b}: the statement there only covers coupling $|B_{m_n}|$, but the proof actually gives $B_{m_n}$. The fact that $m_n$ may be taken up to $o(\log \Phi_n)$ can be inferred from the proof. In fact, $m_n$ could be taken up to $c \log \Phi_n$ with some $c=c(\nu)>0$. The result is then passed to $\UG_n$ by \eqref{e:CMprobsimp}
(see e.g.\ \cite[Corollary~7.17]{vdH17a}).
\end{proof}

\begin{proof}[Proof of Theorem~\ref{t:QLyapCM}]
First note that, by Propositions~\ref{p:CMsimpleconnected}--\eqref{p:CMisUG}, we may assume that $\UG_n$ is connected,
thus fitting the setup of Section~\ref{s:prep}.
Let $U_n(t)$ be the total mass for $\UG_n$ and $U(t)$ the total mass for $\GW$ as in Proposition~\ref{p:couplingCMGW}.
Define
\begin{equation}
U_n^{\circ}(t) := \E_{\cO_n} \left[ \ee^{\int_0^t \xi(X_s) \dd s} \1_{\{\tau_{B^\cc_{t \log t} > t}\}}\right],
\end{equation}
and analogously $U^{\circ}(t)$. By Lemma~\ref{l:firsttruncation} and Proposition~\ref{p:couplingCMGW}, whp as $n \to \infty$, 
\begin{equation}
U_n(t_n) = U_n^{\circ}(t_n) + o(1) 
= U^{\circ}(t_n) + o(1) = U(t_n) + o(1),
\end{equation}
and so \eqref{e:QLyapCM} follows from Theorem~\ref{t:QLyapGWT} after we note that $\nu$ in \eqref{e:defnu} is equal to $E[D_\star - 1]$.
\end{proof}

%%%%%%%%%%%%%%%%%%%%%% APPENDIX %%%%%%%%%%%%%%%%%%%%%%

\appendix

\section{Analysis of $\chi(\rho)$}
\label{s:analysischi}

In this appendix we study the variational problem in \eqref{e:defchiG}. 
In particular, we prove the alternative representations in Proposition~\ref{p:dualrepchi}, 
and we prove Theorem~\ref{t:tildechilargerho}, i.e., we identify for $\varrho\geq 1/\log(d_{\rm min}+1)$ the quantity $\widetilde\chi(\varrho)$ that appears in Theorems \ref{t:QLyapGWT} and \ref{t:QLyapCM} as $\chi_G$ with $G$ the infinite tree with homogeneous degree $d_{\rm min}\in\N \backslash \{1\}$, the smallest degree that has a positive probability in our random graphs. 
In other words, we show that the infimum in \eqref{e:deftildechi} is attained on the infinite tree with the smallest admissible degrees.

It is not hard to understand heuristically why the optimal tree is infinite and has the smallest degree: the first part in \eqref{e:defchiG} (the quadratic energy term coming from the Laplace operator) has a spreading effect and is the smaller the less bonds there are. However, proving this property is not so easy, since the other term (the Legendre transform from the large-deviation term of the random potential) has an opposite effect. In the setting where the underlying graph is $\Z^d$ instead of a tree, this problem is similar to the question whether or not the minimiser has compact support. However, our setting is different because of the exponential growth of balls on trees. We must therefore develop new methods. 

Indeed, we will not study the effect on the principal eigenvalue due to the restriction of a large graph to a subgraph, but rather due to an opposite manipulation, namely, the glueing of two graphs obtained by adding one single edge (or possibly a joining vertex). The effect of such a glueing is examined in Section~\ref{ss:gentheory}. The result will be used in Section~\ref{ss:treesmindeg} to finish the proof of Theorem \ref{t:tildechilargerho}.  Before that, we discuss in Section~\ref{ss:altrepchi} alternative representations for $\chi$ and prove Proposition~\ref{p:dualrepchi}. 

In this section, no probability is involved. We drop $\varrho$ from the notation at many places.

%%%

\subsection{Alternative representations}
\label{ss:altrepchi}

Fix a graph $G=(V,E)$. Recall that $\cP(V)$ denotes the set of probability measures on $V$, and recall that the constant $\chi_G = \chi_G(\varrho) $ in \eqref{e:defchiG} is defined as $\inf_{p \in \cP(V)} [I_E(p) + \varrho J_V(p)]$ with $I,J$ as in \eqref{e:defIJ}. As the next lemma shows, the constant $\widehat{\chi}$ in \eqref{e:defhatchi} can be also represented in terms of $I$,$J$.

\begin{lemma}{\bf [First representation]}
\label{l:hatchidualrep}
For any graph $G=(V,E)$ and any $\Lambda \subset V$,
\begin{equation}
\label{e:hatchidualrep1}
\widehat{\chi}_V(\varrho; G) = \inf_{\substack{ p \in \cP(V)\colon\, \\ \supp(p) \subset \Lambda} } 
\left[ I_E(p) + \varrho J_V(p) \right].
\end{equation}
In particular,
\begin{equation}
\label{e:hatchidualrep2}
\widehat{\chi}_\Lambda(\varrho;G) \geq \widehat{\chi}_V(\varrho ; G) = \chi_G(\varrho).
\end{equation}
\end{lemma}

\begin{proof}
For the proof of \eqref{e:hatchidualrep1}, see \cite[Lemma~2.17]{GM98}. Moroever, \eqref{e:hatchidualrep2} follows from \eqref{e:hatchidualrep1}.
\end{proof}

We next consider the constant $\chi^{\ssup 0}_G$ in \eqref{e:defchi0} for infinite rooted graphs $G = (V,E,\cO)$. 
Note that, by \eqref{e:hatchidualrep1}, $\widehat{\chi}_{B_r}(\varrho; G)$ is non-increasing in $r$. 
With \eqref{e:hatchidualrep2} this implies
\begin{equation}
\label{e:chi0small}
\chi_G^{\ssup 0}(\varrho) = \lim_{r \to \infty} \widehat{\chi}_{B_r}(\varrho ; G) \geq \chi_G(\varrho).
\end{equation}

\begin{lemma}{\bf [Second representation]}
\label{l:chi0large}
For any rooted $G\in \mathfrak{G}_\infty$, $\chi_G(\varrho) = \chi^{\ssup 0}_G(\varrho)$.
\end{lemma}

\begin{proof}
Write $G=(V,E,\cO)$.
By \eqref{e:defchiG},  Lemma~\ref{l:hatchidualrep} and \eqref{e:chi0small}, it suffices to show that, for any $p\in \cP(V)$ and $r \in \N$, there is a $p_r\in\cP(V)$ with support in $B_r$ such that 
\begin{equation}
\label{e:prchi0large1}
\liminf_{r\to\infty} \left\{ I_E(p_r)+\varrho J_V(p_r) \right\} \leq I_E(p)+\varrho J_V(p).
\end{equation}
Simply take
\begin{equation}
p_r(x)= \frac{p(x)\1_{B_r}(x)}{p(B_r)}, \qquad x \in V,
\end{equation} 
i.e., the normalized restriction of $p$ to $B_r$. Then we easily see that 
\begin{equation}
\begin{aligned}
J_V(p_r)-J_V(p) &=-\frac1{p(B_r)} \sum_{x\in B_r} p(x)\log p(x)+\log p(B_r)+\sum_{x\in V} p(x)\log p(x)\\
& \leq  \frac {J_V(p)}{p(B_r)}(1-p(B_r)) \underset{r \to \infty}{\longrightarrow} 0,
\end{aligned}
\end{equation}
where we use $\log p(B_r) \leq 0$ and $p(x) \log p(x)\leq 0$ for every $x$. As for the $I$-term,
\begin{equation}
\begin{aligned}
I_E(p_r)&=\frac1{p(B_r)}\sum_{\{x,y\}\in E \colon x,y\in B_r}\big(\sqrt{p(x)}-\sqrt{p(y)}\,\big)^2\\
&\quad+\frac 12\sum_{\{x, y\} \in E \colon x\in B_r,\, y\in B^\cc_r}\frac {p(x)}{p(B_r)}
\leq \frac{I_E(p)}{p(B_r)}+\frac{d_{\max}}{2} \frac{p(B_{r-1}^\cc)}{p(B_r)},
\end{aligned}
\end{equation}
and therefore
\begin{equation}
I_E(p_r)-I_E(p) \leq \frac{I_E(p)}{p(B_r)}(1-p(B_r))+\frac{d_{\max}}{2} 
\frac{p(B^\cc_{r-1})}{p(B_r)} \underset{r \to \infty}{\longrightarrow} 0.
\end{equation}
\end{proof}

\begin{proof}[Proof of Proposition~\ref{p:dualrepchi}]
The claim follows from Lemmas~\ref{l:hatchidualrep}--\ref{l:chi0large} and \eqref{e:chi0small}.
\end{proof}

%%%

\subsection{Glueing graphs}
\label{ss:gentheory}

Here we analyse the constant $\chi$ of a graph obtained by connecting disjoint graphs. First we show that glueing two graphs together with one additional edge does not decrease the quantity $\chi$:

\begin{lemma}{\bf [Glue two]}
\label{l:compjoin2}
Let $G_i = (V_i, E_i)$, $i=1,2$, be two disjoint connected simple graphs, and let $x_i \in V_i$, $i=1,2$. Denote by ${G}$
the union graph of $G_1$, $G_2$ with one extra edge between $x_1$ and $x_2$, i.e., $ {G} = ( {V},  {E})$ with $ {V} := V_1 \cup V_2$, $ {E} := E_1 \cup E_2 \cup \{(x_1, x_2)\}$. Then
\begin{equation}
\label{e:glueingchi}
\chi_{ {G}} \geq \min \left\{ \chi_{G_1} , \chi_{G_2}  \right\}.
\end{equation}
\end{lemma}

\begin{proof} 
Given $p \in \cP(V)$, let $a_i = p(V_i)$, $i=1,2$, and define $p_i \in \cP(V_i)$ by putting
\begin{equation}
p_i(x) := \begin{cases}
\frac{1}{a_i} p(x) \1_{V_i}(x) & \text{if } a_i> 0,\\
\1_{x_i}(x) & \text{otherwise.}
\end{cases}
\end{equation} 
Straightforward manipulations show that
\begin{equation}
I_E(p) = \sum_{i=1}^2 a_i I_{E_i}(p_i) + \left( \sqrt{p(x_1)} - \sqrt{p(x_2)}\right)^2, 
\qquad J_V(p) = \sum_{i=1}^2 \left[ a_i J_{V_i}(p_i) - a_i \log a_i \right],
\end{equation}
and so
\begin{equation}
I_E(p) + \varrho J_V(p) \geq \sum_{i=1}^2 a_i \Big[ I_{E_i}(p_i) 
+ \varrho J_{V_i}(p_i) \Big]  \geq \min\{\chi_{G_1}, \chi_{G_2}\}.
\end{equation}
The proof is completed by taking the infimum over $p \in \cP(V)$.
\end{proof}

Below it will be useful to define, for $x \in V$,
\begin{equation}
\label{e:altchirest}
\chi^{\ssup{x,b}}_G  = \inf_{\substack{p \in \cP(V),\\p(x)=b}} [I_E(p) + \varrho J_V(p)],
\end{equation}
i.e., a version of $\chi_G$ with ``boundary condition'' $b$ at $x$. It is clear that $\chi^{\ssup{x,b}}_G \geq \chi_G$.

Next we glue several graphs together and derive representations and estimates for the corresponding $\chi$. For $k \in \N$, let $G_i = (V_i, E_i)$, $1 \leq i \leq k$, be a collection of disjoint graphs. Let $x$ be a point not belonging to $\bigcup_{i=1}^k V_i$. For a fixed choice $y_i \in V_i$, $1 \leq i \leq k$, we denote by $\overline{G}_k = (\overline{V}_k, \overline{E}_k)$ the graph obtained by adding an edge from each $y_1, \ldots, y_k$ to $x$, i.e., $\overline{V}_k = V_1 \cup \cdots \cup V_k \cup \{x\}$ and $\overline{E}_k = E_1 \cup \cdots \cup E_k \cup  \{(y_1, \cO),\dots,(y_k,x)\}$.

\begin{lemma}{\bf [Glue many plus vertex]}
\label{l:joiningk}
For any $\varrho>0$, any $k \in \N$, and any $G_i=(V_i, E_i)$, $y_i \in V_i$, $1\leq i \leq k$,
\begin{equation}
\label{e:joiningk}
\begin{aligned}
\chi_{\overline{G}_{k}}  = 
& \inf_{\substack{0 \leq c_i \leq a_i \leq 1, \\ a_1 + \cdots + a_k \leq 1}}
\Big\{
 \sum_{i=1}^k a_i \left(\chi_{G_i}^{\ssup{y_i, c_i/a_i}} - \varrho \log a_i \right) \\
& \;\; + \sum_{i=1}^k \left( \sqrt{c_i} - \sqrt{1-\sum_{i=1}^k a_i} \right)^2
 - \varrho \Big(1-\sum_{i=1}^k a_i \Big) \log \Big(1-\sum_{i=1}^k a_i\Big)
\Big\}.
\end{aligned}
\end{equation}
\end{lemma}

\begin{proof}
The claim follows from straightforward manipulations with \eqref{e:defIJ}.
\end{proof}

Lemma~\ref{l:joiningk} leads to the following comparison lemma. For $j \in \N$, let
\begin{equation}
(G^j_i,y^j_i) = \begin{cases}
(G_i, y_i) & \text{if } i < j,\\
(G_{i+1}, y_{i+1}) & \text{if } i \geq j,
\end{cases}
\end{equation}
i.e., $(G^j_i)_{i\in\N}$ is the sequence $(G_i)_{i\in\N}$ with the $j$-th graph omitted. Let $\overline{G}^j_k$ be the analogue of $\overline{G}_k$ obtained from $G^j_i$, $1 \leq i \leq k$, $i\neq j$, instead of $G_i$, $1 \leq i \leq k$.

\begin{lemma}{\bf [Comparison]}
\label{l:comparejoinedk}
For any $\varrho>0$ and any $k \in \N$,
\begin{equation}
\label{e:joinedk+1}
\begin{aligned}
\chi_{\overline{G}_{k+1}}  
& = \inf_{1 \leq j \leq k+1} \; \inf_{0 \leq c \leq u \leq \tfrac{1}{k+1}} \; 
\inf_{\substack{0 \leq c_i \leq a_i \leq 1, \\ a_1+ \cdots +  a_k \leq 1}}
\Bigg\{ 
(1-u) \Big[
\sum_{i=1}^k a_i \big(\chi_{G_{\sigma_j(i)}}^{\ssup{y_{\sigma_j(i)}, c_i/a_i}} - \varrho \log a_i \big) \\
&  \qquad + \sum_{i=1}^k \left( \sqrt{c_i} - \sqrt{1-\sum_{i=1}^k a_i} \right)^2 
 - \varrho \Big(1-\sum_{i=1}^k a_i \Big) \log \Big(1-\sum_{i=1}^k a_i \Big) \Big]  \\
&  \qquad + u \chi_{G_j}^{\ssup{y_j, c/u}} + \left(\sqrt{c} - \sqrt{(1-u)\Big(1-\sum_{i=1}^k a_i \Big) } \right)^2 \\
& \qquad - \varrho \left[ u \log u + (1-u) \log (1-u) \right] 
\Bigg\}.
\end{aligned}
\end{equation}
Moreover, 
\begin{equation}
\label{e:comparedjoinedk}
\begin{aligned}
\chi_{\overline{G}_{k+1}} 
& \geq \inf_{1\leq j \leq k+1} \inf_{ 0 \leq u \leq \tfrac{1}{k+1}} \Bigg\{ (1-u) \chi_{\overline{G}^j_k}  \\
& \qquad \qquad \qquad \qquad \quad  + \inf_{v \in [0,1]} \Big\{ u \chi_{G_j}^{(y_j, v)} + \mathbbm{1}_{\{u(1+v) \geq 1 \}}\Big[\sqrt{vu} - \sqrt{1-u} \, \Big]^2 \Big\} \\
& \qquad \qquad \qquad \qquad \quad  - \varrho \left[u \log u + (1-u) \log(1-u) \right]
\Bigg\}.
\end{aligned}
\end{equation}
\end{lemma}

\begin{proof}
Note that
\begin{equation}
\begin{aligned}
& \Big\{(c_i, a_i)_{i=1}^{k+1} \colon\, 0 \leq c_i \leq a_i \leq 1, \sum_{i=1}^{k+1} a_i \leq 1 \Big\} 
= \bigcup_{j=1}^{k+1} \left\{
\substack{\Big( (1-u)(c_i, a_i)_{i=1}^{j-1}, (c, u), (1-u)(c_i,a_i)_{i=j}^k \Big)  \colon\,  
\\ 0\leq c \leq u \leq \tfrac{1}{k+1}, 0 \leq c_i \leq a_i \leq 1, \sum_{i=1}^k a_i \leq 1 }\right\},
\end{aligned}
\end{equation}
from which \eqref{e:joinedk+1} follows by straightforward manipulations on \eqref{e:joiningk}. 
To prove \eqref{e:comparedjoinedk}, note that the first term within the square brackets 
in the first two lines of \eqref{e:joinedk+1} equals the term minimised in \eqref{e:joiningk}, 
and is therefore not smaller than $\chi_{\overline{G}^j_k}$.
\end{proof}

\begin{lemma}{\bf [Propagation of lower bounds]}
\label{l:comparelargerho}
If $\varrho>0$, $M \in \R$,  $C >0$ and $k \in \N$ satisfy $\varrho \geq C/\log(k+1)$ and
\begin{equation}
\label{e:assumpcomparelargerho}
\inf_{1 \leq j \leq k+1} \chi_{\overline{G}^j_k}  \geq M, \qquad 
\inf_{1 \leq j \leq k+1} \inf_{v \in [0,1]} \chi_{G_j}^{\ssup{y_j, v}}  \geq M-C,
\end{equation}
then $\chi_{\overline{G}_{k+1}}  \geq M$.
\end{lemma}

\begin{proof}
Dropping some non-negative terms in \eqref{e:comparedjoinedk}, we obtain
\begin{equation}
\begin{aligned}
\chi_{\overline{G}_{k+1}} -M 
& \geq \inf_{0 \leq u \leq 1/(k+1)} \left\{ u \left( \chi_{G_j}^{\ssup{y_j, v}}  - M \right) - \varrho u \log u \right\}\\
& \geq  \inf_{0 \leq u \leq 1/(k+1)} \left\{ u \left( \varrho \log (k+1) - C \right)\right\} \geq 0
\end{aligned}
\end{equation}
by the assumption on $\varrho$.
\end{proof}

\noindent
The above results will be applied in the next section to minimise $\chi$ over families of trees with minimum degrees.

%%%

\subsection{Trees with minimum degrees}
\label{ss:treesmindeg}

Fix $d \in \N$. Let $\mr{\cT}_d$ be an infinite tree rooted at $\cO$ such that the degree of $\cO$ equals $d-1$ and the degree of every other vertex in $\mr{\cT}_d$ is $d$. Let $\mathring{\mathscr{T}}_d^{\ssup 0} = \{\mathring{\cT}_d\}$ and, recursively, let $\mr{\mathscr{T}}_d^{\ssup{n+1}}$ denote the set of  all trees obtained from a tree in $\mr{\mathscr{T}}_d^{\ssup n}$ and a disjoint copy of $\mr{\cT}_d$ by adding an edge between a vertex of the former and the root of the latter. Write $\mr{\mathscr{T}}_d = \bigcup_{n \in\N_0} \mr{\mathscr{T}}_d^{\ssup n}$. Assume that all trees in $\mr{\mathscr{T}}_d$ are rooted at $\cO$.

Recall that $\cT_d$ is the infinite regular $d$-tree. Observe that $\cT_d$ is obtained from $(\mr{\cT}_d, \cO)$ and a disjoint copy $(\mr{\cT}_d', \cO')$ by adding one edge between $\cO$ and $\cO'$. Consider $\cT_d$ to be rooted at $\cO$. Let $\mathscr{T}_d^{\ssup 0} = \{\cT_d\}$ and, recursively, let $\mathscr{T}_d^{\ssup{n+1}}$ denote the set of all trees obtained from a tree in $\mathscr{T}_d^{\ssup n}$ and a disjoint copy of $\mr{\cT}_d$ by adding an edge between a vertex of the former and the root of the latter. Write $\mathscr{T}_d = \bigcup_{n \in\N_0} \mathscr{T}_d^{\ssup n}$, and still consider all trees in $\mathscr{T}_d$ to be rooted at $\cO$. Note that $\mathscr{T}_d^{\ssup n}$ contains precisely those trees of $\mr{\mathscr{T}}_d^{\ssup{n+1}}$ that have $\cT_d$ as a subgraph rooted at $\cO$. In particular, $\mathscr{T}_d^{\ssup n} \subset \mr{\mathscr{T}}_d^{\ssup{n+1}}$ and $\mathscr{T}_d \subset \mr{\mathscr{T}}_d$.

Our objective is to prove the following.

\begin{proposition}{\bf [Minimal tree is optimal]}
\label{p:minimalchiT}
%$\mbox{}$\\
If $\varrho \geq 1/ \log(d+1)$, then 
\[
\chi_{\cT_d}(\varrho)  = \min_{T \in \mathscr{T}_d} \chi_T(\varrho).
\]
\end{proposition}

For the proof of Proposition~\ref{p:minimalchiT}, we will need the following. 

\begin{lemma}{\bf [Minimal half-tree is optimal]}
\label{l:minimalchimrT}
For all $\varrho \in (0,\infty)$, 
\[
\chi_{\mr{\cT}_d}(\varrho)  = \min_{T \in \mr{\mathscr{T}}_d} \chi_T(\varrho).
\]
\end{lemma}

\begin{proof}
Fix $\varrho \in (0,\infty)$. It will be enough to show that
\begin{equation}
\label{e:prminchimrT1}
\chi_{\mr{\cT}_d}  = \min_{T \in \mr{\mathscr{T}}_d^{\ssup n}} \chi_T,\qquad n\in\N_0,
\end{equation}
which we will achieve by induction in $n$. The case $n=0$ is obvious. Assume that \eqref{e:prminchimrT1} holds for some $n\in\N_0$. Any tree $T \in \mr{\mathscr{T}}_d^{\ssup{n+1}}$ can be obtained from a tree $\widetilde{T} \in \mr{\mathscr{T}}_d^{\ssup n}$ and a disjoint copy $\mr{\cT}'_d$ of $\mr{\cT}_d$ by adding an edge between a point $\tilde{x}$ in the vertex set of $\widetilde{T}$ to the root of $\mr{\cT}'_d$. Applying Lemma~\ref{l:compjoin2} together with the induction hypothesis, we obtain
\begin{equation}
\chi_{T}  \geq \min\left\{\chi_{\widetilde{T}} , \chi_{\mr{\cT}'_d}  \right\} \geq \chi_{\mr{\cT}_d},
\end{equation}
which completes the induction step.
\end{proof}

\begin{lemma}{\bf [A priori bounds]}
\label{l:condchiT}
For any $d \in \N$ and any $\varrho \in (0,\infty)$,
\begin{equation}
\chi_{\mr{\cT}_d}(\varrho)  \leq \chi_{\cT_d}(\varrho)  \leq \chi_{\mr{\cT}_d}(\varrho)  + 1.
\end{equation}
\end{lemma}

\begin{proof}
The first inequality follows from Lemma~\ref{l:minimalchimrT}. For the second inequality, note that $\cT_d$ contains as subgraph a copy of $\mr{\cT}_d$, and restrict the minimum in \eqref{e:defchiG} to $p \in \cP(\mr{\cT}_d)$. 
\end{proof}

\begin{proof}[Proof of Proposition~\ref{p:minimalchiT}]
Fix $\varrho \geq 1/\log(d+1)$. It will be enough to show that
\begin{equation}
\label{e:prminchiT1}
\chi_{\cT_d}  = \min_{T \in \mathscr{T}_d^{\ssup n}}\chi_T,\qquad n\in\N_0 .
\end{equation}
We will prove this by induction in $n$. The case $n=0$ is trivial. Assume that, for some $n_0\geq 0$, \eqref{e:prminchiT1} holds for all $n \leq n_0$. Let $T \in \mathscr{T}^{\ssup{n_0+1}}_d$. Then there exists a vertex $x$ of $T$ with degree $k+1 \geq d+1$. Let $y_1, \ldots, y_{k+1}$ be set of neighbours of $x$ in $T$. When we remove the edge between $y_j$ and $x$, we obtain two connected trees; call $G_j$ the one containing $y_j$, and $\overline{G}^j_k$ the other one. With this notation, $T$ may be identified with $\overline{G}_{k+1}$. 

Now, for each $j$, the rooted tree $(G_j,y_j)$ is isomorphic (in the obvious sense) to a tree in $\mr{\mathscr{T}}_d^{\ssup{\ell_j}}$, where $\ell_j \in \N_0$ satisfy $\ell_1 + \cdots + \ell_{k+1} \leq n_0$, while $\overline{G}^j_k$ belongs to $\mathscr{T}_d^{(n_j)}$ for some $n_j \leq n_0$. Therefore, by the induction hypothesis,
\begin{equation}
\label{e:prminchiT2}
\chi_{\overline{G}^j_k}  \geq \chi_{\cT_d},
\end{equation}
while, by \eqref{e:altchirest}, Lemma~\ref{l:minimalchimrT} and Lemma~\ref{l:condchiT},
\begin{equation}
\label{e:prminchiT3}
\inf_{v \in [0,1]} \chi_{G_j}^{(y_j, v)}  \geq \chi_{G_j}  \geq \chi_{\mr{\cT}_d}  \geq \chi_{\cT_d}  -1.
\end{equation}
Thus, by Lemma~\ref{l:compjoin2} applied with $M=\chi_{\cT_d} $ and $C=1$,
\begin{equation}
\chi_{T}  = \chi_{\bar{G}_{k+1}}  \geq \chi_{\cT_d}, 
\end{equation}
which completes the induction step.
\end{proof}

\begin{proof}[Proof of Theorem~\ref{t:tildechilargerho}]
First note that, since $\cT_{d_{\min}}$ has degrees in $\supp(D_g)$, $\widetilde{\chi}(\varrho) \leq \chi_{\cT_{d_{\min}}}(\varrho)$. For the opposite inequality, we proceed as follows. Fix an infinite tree $T$ with degrees in $\supp(D_g)$, and root it at a vertex $\YY$. For $r \in \N$, let $\widetilde{T}_r$ be the tree obtained from $B_r = B^T_r(\YY)$ by attaching to each vertex $x \in B_r$ with $|x| = r$ a number $d_{\min}-1$ of disjoint copies of $(\mathring{\cT}_{d_{\min}}, \cO)$, i.e., adding edges between $x$ and the corresponding roots. Then $\widetilde{T}_r \in \mathscr{T}_{d_{\min}}$ and, since $B_r$ has more out-going edges in $T$ than in $\widetilde{T}_r$, we may check using \eqref{e:hatchidualrep1} that
\begin{equation}
\widehat{\chi}_{B_r}(\varrho ; T) \geq \widehat{\chi}_{B_r}(\varrho; \widetilde{T}_r) 
\geq \chi_{\widetilde{T}_r}(\varrho) \geq \chi_{\cT_{d_{\min}}}(\varrho).
\end{equation}
Taking $r \to \infty$ and applying Proposition~\ref{p:dualrepchi}, we obtain $\chi_T(\varrho) \geq \chi_{\cT_{d_{\min}}}(\varrho)$. Since $T$ is arbitrary, the proof is complete.
\end{proof}

%%%%%%%%%% REFERENCES %%%%%%%%%%%%%%%%%%%%%%%

%%%%%%%%%%%%%%%%%%%%%%%%%%%%%%%%%%%%%%%%%%%

\end{document}